\documentclass[12pt]{article}
\usepackage{amsmath,amssymb,amsthm}
\newtheorem{thm}{Theorem}[section]

 \newtheorem{cor}{Corollary}[section]
 \newtheorem{lem}{Lemma}[section]
 \newtheorem{prop}{Proposition}[section]
 \newtheorem{defn}{Definition}[section]
\newtheorem{rem}{Remark}[section]

\begin{document}
\begin{center}
{\large{\bf The optimal decay estimates on the framework of Besov spaces for generally
dissipative systems}}
\end{center}
\begin{center}
\footnotesize{Jiang Xu}\\[2ex]
\footnotesize{Department of Mathematics, \\ Nanjing
University of Aeronautics and Astronautics, \\
Nanjing 211106, P.R.China,}\\
\footnotesize{jiangxu\underline{ }79@nuaa.edu.cn}\\

\vspace{3mm}

\footnotesize{Faculty of Mathematics, \\ Kyushu University, Fukuoka 812-8581, Japan}\\

\vspace{10mm}

\footnotesize{Shuichi Kawashima}\\[2ex]
\footnotesize{Faculty of Mathematics, \\ Kyushu University, Fukuoka 812-8581, Japan,}\\
\footnotesize{kawashim@math.kyushu-u.ac.jp}\\
\end{center}
\vspace{6mm}

\begin{abstract}
We give a new decay framework for general dissipative hyperbolic system and hyperbolic-parabolic composite system, which allow us to pay less attention on the traditional spectral analysis in comparison with previous efforts. New ingredients lie in the high-frequency and low-frequency decomposition of a pseudo-differential operator and an interpolation inequality related to homogeneous Besov spaces of negative order. Furthermore, we develop the Littlewood-Paley pointwise energy estimates and new time-weighted energy functionals to establish the optimal decay estimates on the framework of spatially critical Besov spaces for degenerately dissipative hyperbolic system of balance laws. Based on the $L^{p}(\mathbb{R}^{n})$ embedding and improved Gagliardo-Nirenberg inequality, the optimal $L^{p}(\mathbb{R}^{n})$-$L^{2}(\mathbb{R}^{n})(1\leq p<2)$ decay rates and $L^{p}(\mathbb{R}^{n})$-$L^{q}(\mathbb{R}^{n})(1\leq p<2\leq q\leq\infty)$ decay rates are further shown. Finally, as a direct application, the optimal decay rates for 3D damped compressible Euler equations are also obtained.
\end{abstract}


\section{Introduction}
In the following we shall consider a general $N$-component hyperbolic system of
balance laws, which is given by
\begin{equation}
U_{t}+\sum_{j=1}^{n}F^{j}(U)_{x_{j}}=G(U) \label{R-E1}
\end{equation}
with the initial data
\begin{equation}
U_{0}=U(0,x),\ \
  x\in\mathbb{R}^{n}\ (n\geq3), \label{R-E2}
\end{equation}
where $U=U(t,x)$ is the unknown $N$-vector valued function taking values in an open convex set $\mathcal{O}_{U}\subset
\mathbb{R}^{N}$ (called the state space). $F^{j}$ and $G$ are given
$N$-vector valued smooth functions on $\mathcal{O}_{U}$. The subscripts $t$ and $x_{j}$ refer to the partial derivatives with respect to $t$ and $x_{j}$, respectively.

As $G(U)\equiv0$, system (\ref{R-E1}) reduces to a system of conservation laws. As we all known, generally,
even for very good initial data, classical solutions may break down in finite
time, due to the appearance of singularities, either discontinuities or blow-up (see,
e.g., \cite{D1}).  As $G(U)\neq0$, system (\ref{R-E1}) describe a great number of non-equilibrium phenomena.
Important examples occur in the study of gas dynamics with relaxation, chemically reactive flows, radiation hydrodynamics,
traffic flows, nonlinear optics and so on. See \cite{Y3} and reference cited therein. In this case, the dissipative mechanism due
to the source term $G(U)$, even if it is sometimes partial, could prevent the formation of singularities. More precisely,
$G(U)$ has, or can be transformed by a
linear transformation into, the form
$$G(U)=\left(
 \begin{array}{c}
                                                                   0 \\
                                                                  g(U) \\
                                                                 \end{array}
                                                               \right),
$$
with $0\in \mathbb{R}^{N_{1}}, g(U)\in \mathbb{R}^{N_{2}}$, where
$N_{1}+N_{2}=N(N_{1}\neq0)$. As observed, the source term is not
present in all the components of (\ref{R-E1}).

To the best of our knowledge, the existence of an entropy and
the so-called Shizuta-Kawashima ([SK]) stability condition, which
was first formulated by \textsc{Shizuta} and the second author \cite{SK} for symmetric hyperbolic-parabolic
systems, could
guarantee the global-in-time existence of classical solutions, either in the Sobolev spaces $H^{l}(\mathbb{R}^{n})$ with higher regularity $l>1+n/2$ (see, e.g, \cite{HN}, \cite{Y})
or in the critical Besov space $B^{\sigma_{c}}_{2,1}(\mathbb{R}^{n})$ with $\sigma_{c}=1+n/2$ (see ourselves \cite{XK}). Let us first recall these efforts for the balance laws (\ref{R-E1}) briefly.

\textsc{Chen, Levermore} \& \textsc{Liu} \cite{CLL} formulated a notion of the entropy for (\ref{R-E1}),
which was a natural extension of the classical entropy due to \textsc{Godunov} \cite{G}, \textsc{Friedrichs}
and \textsc{Lax} \cite{FL} for conservation laws (i.e., $G(U)\equiv 0$). Regretfully, the dissipative
entropy was not strong enough to develop the global existence theory for (\ref{R-E1}). Subsequently,
a technical requirement was imposed on the entropy dissipation, and then the global existence of classical
solutions in a neighborhood of a constant equilibrium $\bar{U}\in \mathbb{R}^{N}$ satisfying $G(\bar{U})=0$
was carried out by \textsc{Hanouzet} and \textsc{Natalini} \cite{HN} in one space dimension and \textsc{Yong} \cite{Y}
in several space dimensions. In the two works, the [SK] condition was all time assumed. Later, the second author and \textsc{Yong} \cite{KY} removed
the technical requirement on the dissipative entropy assumed in \cite{HN,Y} and gave a perfect definition of the entropy for (\ref{R-E1}) (see Definition \ref{defn2.1}) to develop the global-in-time existence of classical solutions, see \cite{KY2}.

However, it should be pointed out that above efforts were established by the basic local-in-time existence theory of\textsc{ Kato} and \textsc{Majda} \cite{K,M} for generally symmetric hyperbolic systems, where the regularity index of spatially Sobolev spaces $H^{l}(\mathbb{R}^n)$
is usually required to be high ($l>1+n/2$). It remains a challenging open problem whether classical solutions of (\ref{R-E1}) still preserve
the global well-posedness and stability in the case of the critical regularity $\sigma_{c}=1+n/2$. To solve it, in the very recent work \cite{XK}, we have already introduced a class of mixed space-time Besov spaces (usually referred to the Chemin-Lerner spaces) which was initialed by Chemin and Lerner \cite{CL}, and developed an elementary fact that indicates the relation between homogeneous and inhomogeneous Chemin-Lerner spaces, which enable us to capture the degenerate dissipation rates generated from partial source terms and obtain the global well-posedness on the framework of spatially Besov spaces $B^{1+n/2}_{2,1}(\mathbb{R}^{n})$.

Concerning the large-time asymptotic stability of solutions,
\textsc{Umeda}, the second author and \textsc{Shizuta} \cite{UKS} first employed the pointwise energy estimates in Fourier spaces
to show the time-decay property for a general class of symmetric hyperbolic-parabolic composition system (including the symmetric hyperbolic system), provided that the initial data belong to $L^2(\mathbb{R}^{n})\cap L^p(\mathbb{R}^{n})(1\leq p<2)$, where the low-frequency and high-frequency integrals were performed respectively, based on the Hausdorff-Young inequality. Subsequently, the second author \cite{Ka} studied the corresponding nonlinear systems of a hyperbolic-parabolic composite type and obtained the decay estimates of classical solutions in $H^{l}(\mathbb{R}^{n})\cap L^p(\mathbb{R}^{n}) (l>2+n/2, 1\leq p<2)$, where he found that the overall decay rate in these norms was the heat kernel rate.
\textsc{Ruggeri} and \textsc{Serre} \cite{RS} showed that the constant equilibrium state $\bar{U}$ is time asymptotically $L^2$-stable for (\ref{R-E1}), by constructing the Lyapunov functionals. Moreover, it was proved by \textsc{Bianchini, Hanouzet} \& \textsc{Natalini} \cite{BHN} that classical solutions approach the constant equilibrium state in the $L^{p}$-norm at the rate $O(t^{-\frac{n}{2}(1-\frac{1}{p})})$, as $t\rightarrow\infty$, for $p\in[\min\{n,2\},\infty]$. Due to the strong quasilinearity of (\ref{R-E1}), however, the optimal $L^p$ decay rates are valid only for a few derivatives of solutions.
To make up for the deficiency, the second author and \textsc{Yong} \cite{KY2} employed the time-weighted energy functional which was first developed in \cite{Ma} for compressible Navier-Stokes equations to obtain the desired $L^p$ decay estimates.

As shown by \cite{XK}, the harmonic analysis technique allow to reduce significantly the regularity requirement and establish the
global-in-time existence of classical solutions for (\ref{R-E1}) in spatially critical Besov spaces. As the continuation study,
an interesting problem is how fast does the solution decay on the new functional framework.

The paper is organized as follows. In Sect.~\ref{sec:2}, we recall the entropy and symmetrization about the hyperbolic system as well as the [SK] stability condition. Main results will also be stated in this section. In Sect.~\ref{sec:3}, we introduce the Littlewood-Paley decomposition theory and give the definition of Besov spaces as well as some useful facts in Besov spaces. Sect.~\ref{sec:4} is devoted to develop the L-P pointwise energy estimates, which deduce the optimal decay estimates for the symmetric hyperbolic system on the framework of Besov spaces. In Sect.~\ref{sec:5}, we use the modified time-weighted energy approaches related to the low-frequency and high-frequency decomposition along with interpolation inequalities to deduce the corresponding decay estimates for the nonlinear symmetric hyperbolic system effectively.  As a
direct application of our results, in Sect.~\ref{sec:6}, we obtain the decay estimates for the damped 3D compressible Euler equations. The paper will be end with two Appendixes. In Appendix A, the decay framework for linearized dissipative systems, either the damped hyperbolic system or hyperbolic-parabolic composite system, will be shown, provided that the initial data belong to $L^2(\mathbb{R}^n)\cap \dot{B}^{-s}_{2,\infty}(\mathbb{R}^n)(s>0)$. It should be a great improvement in comparison with the framework in \cite{UKS} which was given by the second author etc. thirty years ago, since $L^{1}(\mathbb{R}^n)\hookrightarrow \dot{B}^{0}_{1,\infty}(\mathbb{R}^n)\hookrightarrow\dot{B}^{-n/2}_{2,\infty}(\mathbb{R}^n)$. Actually, this is the main motivation for
the present work. In Appendix B, for the convenience of readers, some interpolation inequalities related to Besov spaces, for instance, $L^{p}(\mathbb{R}^{n})$ embedding and improved Gagliardo-Nirenberg-Sobolev inequality, will be presented.

\section{Entropy, [SK] condition and main results}\setcounter{equation}{0}\label{sec:2}
It is convenient to state main results of this
paper, we first review the dissipative entropy and the stability
condition. To begin with, let us
set
\begin{eqnarray*}
\mathcal{M}=\{\psi\in\mathbb{R}^{N}: \langle\psi,G(U)\rangle=0\ \
\mbox{for any}\ U\in \mathcal{O}_{U}\}.
\end{eqnarray*}
Then $\mathcal{M}$ is a subset of $\mathbb{R}^{N}$ with $\mathrm{dim}
\mathcal{M}=N_{1}$. In the discrete kinetic theory, $\mathcal{M}$ is
called the space of summational (collision) invariants. From the
definition of $\mathcal{M}$, we have
\begin{eqnarray*}
G(U)\in\mathcal{M}^{\top}(\mbox{the orthogonal complement of}\
\mathcal{M}),\ \mbox{for any}\ U\in \mathcal{O}_{U}.
\end{eqnarray*}
Furthermore, corresponding to the orthogonal decomposition
$\mathbb{R}^{N}=\mathcal{M}\oplus\mathcal{M}^{\top}$, we may write
$U\in\mathbb{R}^{N}$ as
\begin{eqnarray*}
U=\left(
    \begin{array}{c}
      U_{1} \\
      U_{2} \\
    \end{array}
  \right)
\end{eqnarray*}
such that $U\in\mathcal{M}$ holds if and only if $U_{2}=0$. We
denote by $\mathcal{E}$ the set of equilibrium state for the balance
laws (\ref{R-E1}):
\begin{eqnarray*}
\mathcal{E}=\{U\in \mathcal{O}_{U}: G(U)=0\}.
\end{eqnarray*}

\subsection{Entropy and symmetrization}
In \cite{KY}, an entropy for (\ref{R-E1}) is defined as follows.
\begin{defn}\label{defn2.1}
Let\ \ $\eta=\eta(U)$ be a smooth function defined in a convex open
set $\mathcal{O}_{U}\subset\mathbb{R}^{N}$. Then $\eta=\eta(U)$ is
called an entropy for the balance laws (\ref{R-E1}) if the following
statements hold:
\begin{itemize}
\item [$(\bullet)$] $\eta=\eta(U)$ is strictly convex in $\mathcal{O}_{U}$
in the sense that the Hessian $D^2_{U}\eta(U)$ is positive definite
for\ $U\in \mathcal{O}_{U}$;
\item [$(\bullet)$] $D_{U}F_{j}(U)(D^2_{U}\eta(U))^{-1}$ is symmetric for $U\in \mathcal{O}_{U}$ and $j=1,...,n;$
\item [$(\bullet)$] $U\in\mathcal{E}$ if and only if $(D_{U}\eta(U))^{\top}\in
\mathcal{M}$;
\item [$(\bullet)$] For $U\in\mathcal{E}$, the matrix
$D_{U}G(U)(D^2_{U}\eta(U))^{-1}$ is symmetric and nonpositive
definite, and its null space coincides with $\mathcal{M}$.
\end{itemize}
\end{defn}
\noindent Here and below, $D_{U}$ stands for the (row) gradient
operator with respect to $U$. Let $\eta(U)$ be the entropy and set
\begin{eqnarray}W(U)=(D_{U}\eta(U))^{\top}. \label{R-E3}\end{eqnarray}
It was shown in \cite{KY} that the mapping $W=W(U)$ is a
diffeomorphism from $\mathcal{O}_{U}$ onto its range
$\mathcal{O}_{W}$. Let $U=U(W)$ be the inverse mapping which is also
a diffeomorphism from $\mathcal{O}_{W}$ onto its range
$\mathcal{O}_{U}$. Then (\ref{R-E1}) can be rewritten as

\begin{eqnarray}
A^{0}(W)W_{t}+\sum^{n}_{j=1}A^{j}(W)W_{x_{j}}=H(W) \label{R-E4}
\end{eqnarray}
with $$A^{0}(W)=D_{W}U(W),$$
$$A^{j}(W)=D_{W}F^{j}(U(W))=D_{U}F^{j}(U(W))D_{W}U(W),$$
$$H(W)=G(U(W)).$$
Moreover, define
$$L(W):=-D_{W}H(W)=-D_{U}G(U(W))D_{W}U(W).$$
In terms of (\ref{R-E3}), we have
$D_{W}U(W)=D^2_{U}\eta(U(W))^{-1}$. Then it was shown in \cite{KY}
that (\ref{R-E4}) is a \textit{symmetric dissipative} system in the
sense defined as follows.

\begin{defn}\label{defn2.2}
The system (\ref{R-E4}) is called symmetric dissipative if the
following statements hold:
\begin{itemize}
\item [$(\bullet)$] $A^{0}(W)$ is symmetric and positive definite for $W\in\mathcal{O}_{W}$;
\item [$(\bullet)$] $A^{j}(W)$ is symmetric for $W\in \mathcal{O}_{W}$ and $j=1,...,n;$
\item [$(\bullet)$] $H(W)=0$ if and only if $W\in \mathcal{M}$;
\item [$(\bullet)$] For $W\in \mathcal{M}$, the matrix
$L(W)$ is symmetric and nonnegative definite, and its null space
coincides with $\mathcal{M}$.
\end{itemize}
\end{defn}

Meantime, $H(W)$ has a useful expression for any fixed constant state $\bar{W}\in\mathcal{M}$ (see \cite{KY}):
\begin{eqnarray}
H(W)=-LW+r(W), \label{R-E5}
\end{eqnarray}
where $L=L(\bar{W}),\ r(W)\in \mathcal{M}^{\top}$ for all
$W\in\mathcal{O}_{W}$ and satisfies
\begin{eqnarray}
|r(W)|\leq C|W-\bar{W}||(I-\mathcal{P})W| \label{R-E6}
\end{eqnarray}
for $W\in\mathcal{O}_{W}$ close to $\bar{W}$, where $I$ the identity
mapping on $\mathbb{R}^{N}$ and $\mathcal{P}$ the orthogonal
projection onto $\mathcal{M}$.

In order to obtain the effective decay estimates for orthogonal part $(I-\mathcal{P})W$, it is convenient to
transform (\ref{R-E4}) into a symmetric dissipative system of \textit{normal form} in the sense
defined below.

\begin{defn}\label{defn2.3}
The symmetric dissipative system (\ref{R-E4}) is said to be of the
normal form if $A^{0}(W)$ is block-diagonal associated with
orthogonal decomposition
$\mathbb{R}^{N}=\mathcal{M}\oplus\mathcal{M}^{\bot}$.
\end{defn}

Use the partition as $$U=\left(
                           \begin{array}{c}
                             U_{1} \\
                             U_{2} \\
                           \end{array}
                         \right),\ \ \ W=\left(
                           \begin{array}{c}
                             W_{1} \\
                             W_{2} \\
                           \end{array}
                         \right)
$$
associated with the orthogonal decomposition
$\mathbb{R}^{N}=\mathcal{M}\oplus\mathcal{M}^{\bot}$. Consider
the mapping $U\rightarrow V$ defined by
$$V=\left(
                           \begin{array}{c}
                             V_{1} \\
                             V_{2} \\
                           \end{array}
                         \right)=\left(
                           \begin{array}{c}
                             U_{1} \\
                             W_{2} \\
                           \end{array}
                         \right)$$
with $W_{2}=(D_{U_{2}}\eta(U))^{\top}$.  This is a diffeomorphism
from $\mathcal{O}_{U}$ onto its range $\mathcal{O}_{V}$. Denote by
$U=U(V)$ the inverse mapping which is a diffeomorphism from
$\mathcal{O}_{V}$ onto its range $\mathcal{O}_{U}$. Hence, $W=W(V)$
is the diffeomorphism composed by $W=W(U)$ and $U=U(V)$. The
straightforward calculation gives
\begin{eqnarray}
\tilde{A}^{0}(V)V_{t}+\sum^{n}_{j=1}\tilde{A}^{j}(V)V_{x_{j}}=\tilde{H}(V)\label{R-E7}
\end{eqnarray}
with $$\tilde{A}^{0}(V)=(D_{V}W)^{\top}A^{0}(W)D_{V}W,$$
$$\tilde{A}^{j}(V)=(D_{V}W)^{\top}A^{j}(W)D_{V}W,$$
$$\tilde{H}(V)=(D_{V}W)^{\top}H(W),$$
where $W$ is evaluated at $W(V)$. It was shown by \cite{KY2} that (\ref{R-E7}) is a symmetric dissipative system of the
normal form. Precisely,
\begin{thm} \label{thm2.1}
The system (\ref{R-E7}) is a symmetric dissipative system of the
normal form and $\tilde{H}(V)=H(W)$. It holds $W\in\mathcal{M}$ if
and only if $V\in\mathcal{M}$ between the variables $W$ with $V$.
Furthermore, the matrix $\tilde{L}(V):=-D_{V}\tilde{H}(V)$ can be
expressed as
$$\tilde{L}(V)=(D_{V}W)^{\top}L(W)D_{V}W$$ and satisfies
$\tilde{L}(V)=L(W)$ if $V\in \mathcal{M}$ (i.e., $W\in
\mathcal{M}$).
\end{thm}

As a simple consequence, we have an analogue of (\ref{R-E5})-(\ref{R-E6}) that $\tilde{H}(V)$ has the following expression for any fixed constant state $\bar{V}\in\mathcal{M}$ (i.e., $\bar{U}\in\mathcal{M}$):
\begin{eqnarray}
\tilde{H}(V)=-LV+\tilde{r}(V),\label{R-E8}
\end{eqnarray}
where $L=L(\bar{W}),\ \tilde{r}(V)\in \mathcal{M}^{\top}$ for all
$V\in\mathcal{O}_{V}$. Furthermore,
\begin{eqnarray}
|\tilde{r}(V)|\leq
C|V-\bar{V}||(I-\mathcal{P})V| \label{R-E9}
\end{eqnarray}
for $V\in\mathcal{O}_{V}$ close to $\bar{V}$.

\subsection{[SK] condition}
As far as we known, the [SK] stability condition play a key role in the study of generally dissipative systems,
which was first formulated by \textsc{Shizuta} and the second author \cite{SK} to deduce the decay estimates of solutions.
Here, we formulate the [SK] condition for (\ref{R-E7}),
since we shall deal with the symmetric dissipative system of normal form
in the subsequent analysis. Let $\bar{V}\in\mathcal{M}$ be a
constant state and consider the linearized form of (\ref{R-E7}) at
$V=\bar{V}$:

\begin{equation}
\tilde{A}^{0}V_{t}+\sum_{j=1}^{n}\tilde{A}^{j}V_{x_{j}}+LV=0,\label{R-E10}
\end{equation}
where $\tilde{A}^{0}=\tilde{A}^{0}(\bar{V}),\
\tilde{A}^{j}=\tilde{A}^{j}(\bar{V})$ and $L=L(\bar{V})$. Taking the
Fourier transform on (\ref{R-E10}) with respect to $x\in
\mathbb{R}^{n}$, we obtain
\begin{equation}
\tilde{A}^{0}\hat{V}_{t}+i|\xi|\tilde{A}(\omega)\hat{V}+L\hat{V}=0,
\label{R-E11}
\end{equation}
where $\tilde{A}(\omega):=\sum_{j=1}^{n}\tilde{A}^{j}\omega_{j}$
with $\omega=\xi/|\xi|\in \mathbb{S}^{n-1}$ (the unit sphere). Let $\lambda=\lambda(i\xi)$ be the eigenvalues of
(\ref{R-E11}), which solves the characteristic equation
$$\mathrm{det}(\lambda\tilde{A}^{0}+i|\xi|\tilde{A}(\omega)+L)=0.$$
Then the stability condition for (\ref{R-E10}) is formulated as follows.

\begin{defn}\label{defn2.4}
The symmetric form (\ref{R-E10}) satisfies the stability condition at
$\bar{V}\in\mathcal{M}$ if the following holds true: Let $\phi\in
\mathbb{R}^{N}$ satisfies $\phi\in\mathcal{M}$ (i.e., $L\phi=0$) and
$\lambda\tilde{A}^{0}+\tilde{A}(\omega)\phi=0$ for some
$(\lambda,\omega)\in \mathbb{R}\times\mathbb{S}^{n-1}$, then
$\phi=0$.
\end{defn}

Actually, the [SK] condition was first formulated in \cite{SK} for a general class of
symmetric hyperbolic-parabolic composite systems including the present
symmetric hyperbolic system (\ref{R-E10}). Moreover, there is an equivalent characterization for the [SK] condition.

\begin{thm}\label{thm2.2}(\cite{SK}) The following statements are
equivalent to each other.
\begin{itemize}
\item [$(\bullet)$] The system (\ref{R-E5}) satisfies the stability condition at $\bar{V}\in\mathcal{M}$;
\item [$(\bullet)$] $\mathrm{Re}\lambda(i\xi)<0$ for $\xi\neq0$;
\item [$(\bullet)$] There is a constant $c>0$ such that $\mathrm{Re}\lambda(i\xi)\leq-c|\xi|^2/(1+|\xi|^2)$ for $\xi\in\mathbb{R}^{n}$;
\item [$(\bullet)$] There is an $N\times N$ matrix $\tilde{K}(\omega)$ depending smooth on $\omega\in\mathbb{S}^{n-1}$ satisfying the
properties:
\begin{itemize}
\item [(i)] $\tilde{K}(-\omega)=-\tilde{K}(\omega)$ for
$\omega\in\mathbb{S}^{n-1}$; \item [(ii)] $\tilde{K}(\omega)\tilde{A}^{0}$
is
skew-symmetric for $\omega\in\mathbb{S}^{n-1}$;
\item [(iii)]$[\tilde{K}(\omega)\tilde{A}(\omega)]'+L$ is positive definite for
$\omega\in\mathbb{S}^{n-1}$, where $[X]'$ denotes the symmetric part of the matrix $X$.
\end{itemize}
\end{itemize}
\end{thm}

\subsection{Main results}
 In this paragraph, we state main results of this paper. First, we recall the global-in-time existence of classical solutions to the Cauchy problem (\ref{R-E1})-(\ref{R-E2}), the interesting reader is referred to \cite{XK} for the detailed proof. Here and below, denote by $\sigma_{c}$ the critical regularity $\sigma_{c}:=1+n/2$.

\begin{thm}(\cite{XK})\label{thm2.3}
Suppose the balance laws (\ref{R-E1}) admits an entropy defined as
Definition \ref{defn2.1} and the corresponding symmetric system
(\ref{R-E10}) satisfies the stability condition at $\bar{V}\in
\mathcal{M}$, where $\bar{V}$ is the constant state corresponding to
$\bar{U}$. There exists a positive constant $\delta_{0}$ such that
if
\begin{eqnarray*}
\|U_{0}-\bar{U}\|_{B^{\sigma_{c}}_{2,1}(\mathbb{R}^{n})}\leq \delta_{0},
\end{eqnarray*}
then the Cauchy problem (\ref{R-E1})-(\ref{R-E2}) has a unique
global solution $U\in \mathcal{C}^{1}(\mathbb{R}^{+}\times
\mathbb{R}^{n})$ satisfying
\begin{eqnarray*}
U-\bar{U}\in
\widetilde{\mathcal{C}}(B^{\sigma_{c}}_{2,1}(\mathbb{R}^{n}))\cap
\widetilde{\mathcal{C}}^1(B^{\sigma_{c}-1}_{2,1}(\mathbb{R}^{n})).
\end{eqnarray*}
Moreover, there exist two positive constants $C_{0},$ and $\mu_{0}$ such that the following energy inequality holds
\begin{eqnarray}
&&\|U-\bar{U}\|_{\widetilde{L}^\infty(B^{\sigma_{c}}_{2,1}(\mathbb{R}^{n}))}
+\mu_{0}\Big(\|(I-\mathcal{P})V\|_{\widetilde{L}^2(B^{\sigma_{c}}_{2,1}(\mathbb{R}^{n}))}
+\|\nabla
V\|_{\widetilde{L}^2(B^{\sigma_{c}-1}_{2,1}(\mathbb{R}^{n}))}\Big)
\nonumber\\&\leq&
C_{0}\|U_{0}-\bar{U}\|_{B^{\sigma_{c}}_{2,1}(\mathbb{R}^{n})},
\label{R-E12}
\end{eqnarray}
where $\mathcal{P}$ is the orthogonal projection onto $\mathcal{M}$.
\end{thm}

Let us explain the notation of functional spaces appearing in Theorem \ref{thm2.3}.
Define
$$\widetilde{\mathcal{C}}_{T}(B^{\sigma_{c}}_{2,1}):=\widetilde{L}^{\infty}_{T}(B^{\sigma_{c}}_{2,1})\cap\mathcal{C}([0,T],B^{\sigma_{c}}_{2,1})$$
and
$$\widetilde{\mathcal{C}}^1_{T}(B^{\sigma_{c}-1}_{2,1}):=\{f\in\mathcal{C}^1([0,T],B^{\sigma_{c}-1}_{2,1})|\partial_{t}f\in\widetilde{L}^{\infty}_{T}(B^{\sigma_{c}-1}_{2,1})\},$$
where the index $T$ can be omitted when $T=+\infty$. For the definition of mixed space-time Besov spaces (Chemin-Lerner spaces), see, e.g., \cite{BCD,CL}.

A natural question follows. How does the large-time behavior of solutions exhibit on the framework of spatially critical Besov spaces?
Following from the Gagliardo-Nirenberg-Sobolev inequality and complex interpolation inequality related to Besov spaces, the asymptotics
of solutions $U$ and $(I-\mathcal{P})U$ have been shown (see \cite{XK}), however, the definite decay rates are not available, so we intend to further
answer the question.

Based on the decay framework $L^2(\mathbb{R}^{n})\cap L^{p}(\mathbb{R}^{n})(1\leq p<2)$ for linearized dissipative systems (see \cite{UKS}), the second author \cite{Ka} further considered the generally hyperbolic-parabolic composite systems and obtained the optimal decay estimates in $H^{l}(\mathbb{R}^{n})\cap L^{1}(\mathbb{R}^{n})(l>2+n/2)$. The main analysis ingredients involved the pointwise
energy estimates and low-frequency and high-frequency integrals in Fourier spaces. $L^q(q\geq2)$ estimates
were then obtained by the Gagliardo-Nirenberg interpolation inequality. This effort has been developed great, for instance, by Hoff and Zumbrun \cite{HZ} for compressible Navier-Stokes equations. They performed the more elaborate spectral analysis on the Green's matrix and a number of Fourier multiplier and Paley-Wiener
theory, which allow to produce the $L^q$ decay rates for any $1\leq q\leq\infty$ in $H^{l}(\mathbb{R}^{n})\cap L^{1}(\mathbb{R}^{n})(l>2+n/2)$.

In this paper, we give a new decay framework for linearized dissipative systems in $L^2(\mathbb{R}^{n})\cap\dot{B}^{-s}_{2,\infty}(\mathbb{R}^{n})(s>0)$, which
improves the classical framework in \cite{UKS} essentially, since $L^{1}(\mathbb{R}^{n})\hookrightarrow\dot{B}^{0}_{1,\infty}(\mathbb{R}^{n})\hookrightarrow\dot{B}^{-\frac{n}{2}}_{2,\infty}(\mathbb{R}^{n})$. Let us sketch the technical obstruction and the strategy to overcome it. Firstly, due to the loss of integrality of $\dot{B}^{-s}_{2,\infty}(\mathbb{R}^{n})$, we have to go considerably beyond the usual low-frequency and high-frequency integrals based on the Hausdorff-Young inequality in \cite{Ka,UKS}. A key observation is that the symbol of pseudo-differential operator $(1-\Delta)^{-\frac{1}{2}}\nabla$ behaves $1$ when $|\xi|$ goes to infinity
and $\xi$ when $|\xi|$ goes to zero, which motivates the different considerations on the high-frequency part and low-frequency part of solutions. By virtue of the unit decomposition, we can see the high-frequency part of solutions decays exponentially in time, however, the low-frequency part of solutions need to be dealt with more carefully. Fortunately, thanks to the interpolation inequality related to the Besov space $\dot{B}^{-s}_{2,\infty}(\mathbb{R}^{n})$ and the fact that the $\dot{B}^{-s}_{2,\infty}$-norm of solutions is preserved along time evolution for the linearized dissipative equations, the desired differential inequality on the low-frequency of solutions is achieved, which leads to the desired decay estimates. We would like to point out such decay framework for generally dissipative systems is new, which allows to pay less attention on the traditional spectral analysis, e.g., as in \cite{BHN,HZ,Ka,UKS}, see Appendix A.

Secondly, the low-frequency and high-frequency ideas further inspire us to use the Littlewood-Paley decomposition and do more elaborate analysis, that is, to
develop the L-P pointwise energy estimates and obtain the optimal decay estimates on the framework of Besov spaces, either the inhomogeneous case or the homogeneous case. As shown by Sect.~\ref{sec:4}, we see that the proof ideas on the low-frequency part are quite different due to the restriction use of the interpolation inequality related to $\dot{B}^{-s}_{2,\infty}(\mathbb{R}^{n})$.
To treat the difficulty due to the strong quasilinearity of (\ref{R-E1}), as in \cite{KY2}, the time-weighted energy approach is mainly used. However, in order to take care of the topological relation between inhomegeous Besov spaces and homogeneous Besov spaces, the new energy functional contains different time-weighted norms according to the derivative index, which is used to overcome the technical difficulty in subsequent nonlinear analysis. Here, there is a little surprising that the energy functional can be regarded as the improvement of that in \cite{KY2}, where the derivative index of solutions can take all values in the total interval $[0,\sigma_{c}-1]$ rather than nonnegative integers only. To get round the obstruction originated from partially dissipative structures, the time-weighted energy functional for $(I-\mathcal{P})U$
is also introduced.  Furthermore, the Gagliardo-Nirenberg-Sobolev inequality as in \cite{N} is well improved, which allows to the case of fractional derivatives, see Lemma \ref{lem8.4}. In terms of the frequency-localization Duhamel principle, the nonlinear decay estimates are ultimately constructed by using the low-frequencey and high-frequency decomposition methods and the improved Gagliardo-Nirenberg-Sobolev inequality, see Sect.~\ref{sec:5} for details.

Denote $\Lambda^{\ell}f:=\mathcal{F}^{-1}|\xi|^{\ell}\mathcal{F}f$. Our main results are stated as follows.

\begin{thm}\label{thm2.4}
Let $U(t,x)$ be the global classical solution of Theorem \ref{thm2.3}. If further the initial data $U_{0}\in \dot{B}^{-s}_{2,\infty}(\mathbb{R}^{n})(0<s\leq n/2)$ and
$$E_{0}:=\|U_{0}-\bar{U}\|_{B^{\sigma_{c}}_{2,1}(\mathbb{R}^{n})\cap\dot{B}^{-s}_{2,\infty}(\mathbb{R}^{n})}$$
is sufficiently small. Then the classical solution $U(t,x)$  satisfies the following decay estimates
\begin{eqnarray}
\|\Lambda^{\ell}[U(t,\cdot)-\bar{U}]\|_{X_{1}(\mathbb{R}^{n})}\lesssim E_{0}(1+t)^{-\frac{s+\ell}{2}} \label{R-E13}
\end{eqnarray}
for $0\leq\ell\leq \sigma_{c}-1$, where $X_{1}:=B^{\sigma_{c}-1-\ell}_{2,1}$ if $0\leq\ell<\sigma_{c}-1$ and $X_{1}:=\dot{B}^{0}_{2,1}$ if $\ell=\sigma_{c}-1$;
\begin{eqnarray}
\|\Lambda^{\ell}(I-\mathcal{P})U(t,\cdot)\|_{X_{2}(\mathbb{R}^{n})}\lesssim E_{0}(1+t)^{-\frac{s+\ell+1}{2}} \label{R-E14}
\end{eqnarray}
for $0\leq\ell\leq \sigma_{c}-2$, where $X_{2}:=B^{\sigma_{c}-2-\ell}_{2,1}$ if $0\leq\ell<\sigma_{c}-2$ and $X_{2}:=\dot{B}^{0}_{2,1}$ if $\ell=\sigma_{c}-2$.
\end{thm}
Note that the $L^p(\mathbb{R}^{n})$ embedding in Lemma \ref{lem8.5}, we can obtain the optimal decay rates.

\begin{thm}\label{thm2.5}
Let $U(t,x)$ be the global classical solution of Theorem \ref{thm2.3}. If further the initial data $U_{0}\in L^p(\mathbb{R}^{n})(1\leq p<2)$ and
$$\widetilde{E}_{0}:=\|U_{0}-\bar{U}\|_{B^{\sigma_{c}}_{2,1}(\mathbb{R}^{n})\cap L^p(\mathbb{R}^{n})}$$
is sufficiently small. Then the classical solutions $U(t,x)$  satisfies the following optimal decay estimates
\begin{eqnarray}
\|\Lambda^{\ell}[U(t,\cdot)-\bar{U}]\|_{X_{1}(\mathbb{R}^{n})}\lesssim E_{0}(1+t)^{-\gamma_{p,2}-\frac{\ell}{2}} \label{R-E15}
\end{eqnarray}
for $0\leq\ell\leq \sigma_{c}-1$,
and
\begin{eqnarray}
\|\Lambda^{\ell}(I-\mathcal{P})U(t,\cdot)\|_{X_{2}(\mathbb{R}^{n})}\lesssim E_{0}(1+t)^{-\gamma_{p,2}-\frac{\ell+1}{2}} \label{R-E16}
\end{eqnarray}
for $0\leq\ell\leq \sigma_{c}-2$, where $X_{1}$ and $X_{2}$ are the same space notations as in Theorem \ref{thm2.4}. We denote by $\gamma_{p,2}:=\frac{n}{2}(\frac{1}{p}-\frac{1}{2})$ the $L^p(\mathbb{R}^{n})$-$L^2(\mathbb{R}^{n})$ decay rates for the heat kernerl.
\end{thm}

\begin{rem}\label{rem2.1}
Let us mention that Theorems \ref{thm2.4}-\ref{thm2.5} exhibit the various decay rates of solution and its higher order derivatives. The harmonic analysis allow to reduce significantly the regularity requirements on the initial data in comparison with the previous works \cite{BHN,HZ,Ka,UKS}.
It is worth noting that the derivative index $\ell$ can take values in the interval, for example, $[0,\sigma_{c}-1]$ rather than nonnegative integers only. Additionally, the decay of the orthogonal part $(I-\mathcal{P})U$ of solution is faster at half rate among all the components of solution.
\end{rem}

\begin{rem}\label{rem2.2}
The proofs of Theorems \ref{thm2.4}-\ref{thm2.5} consists several steps. The first step is to develop L-P pointwise energy estimates, which lead to the decay estimates for the linearized dissipative hyperbolic system, see Sect.~\ref{sec:4}. The next step is to develop the frequency-localization Duhamel principle
and time-weighted energy functionals to construct the decay estimates for nonlinear dissipative system. A new point is that the low-frequency and high-frequency decomposition method and improved Gagliardo-Nirenberg-Sobolev inequality will be firstly used. To close the Lyapunov inequality related to time-weighted energy functionals, the last step is to construct the decay estimates for the orthogonal part $(I-\mathcal{P})U$ of solution based on the symmetric hyperbolic system of norm form.
\end{rem}

As the immediate consequence of Theorems \ref{thm2.4}-\ref{thm2.5}, the optimal decay rates in the usual $L^2(\mathbb{R}^{n})$ space are available.
\begin{cor}\label{cor2.1}
Let $U(t,x)$ be the global classical solutions of Theorem \ref{thm2.3}.
\begin{itemize}
\item [(i)]  If $E_{0}$ is sufficiently small, then
\begin{eqnarray}
\|\Lambda^{\ell}[U(t,\cdot)-\bar{U}]\|_{L^2(\mathbb{R}^{n})}\lesssim E_{0}(1+t)^{-\frac{\ell+s}{2}},  \  0\leq\ell\leq\sigma_{c}-1; \label{R-E17}
\end{eqnarray}
\begin{eqnarray}
\|\Lambda^{\ell}(I-\mathcal{P})U(t,\cdot)\|_{L^2(\mathbb{R}^{n})}\lesssim  E_{0}(1+t)^{-\frac{s+\ell+1}{2}}, \  0\leq\ell\leq \sigma_{c}-2. \label{R-E18}
\end{eqnarray}

\item [(ii)] If $\widetilde{E}_{0}$ is sufficiently small, then
\begin{eqnarray}
\|\Lambda^{\ell}[U(t,\cdot)-\bar{U}]\|_{L^2(\mathbb{R}^{n})}\lesssim \widetilde{E}_{0}(1+t)^{-\gamma_{p,2}-\frac{\ell}{2}}, \ 0\leq\ell\leq\sigma_{c}-1; \label{R-E19}
\end{eqnarray}
\begin{eqnarray}
\|\Lambda^{\ell}(I-\mathcal{P})U(t,\cdot)\|_{L^2(\mathbb{R}^{n})}\lesssim \widetilde{E}_{0}(1+t)^{-\gamma_{p,2}-\frac{\ell+1}{2}},\  0\leq\ell\leq \sigma_{c}-2. \label{R-E20}
\end{eqnarray}
\end{itemize}
\end{cor}

In terms of improved Gagliardo-Nirenberg-Sobolev inequality (Lemma \ref{lem8.4}), furthermore, we can obtain the optimal $L^{p}(\mathbb{R}^{n})$-$L^{q}(\mathbb{R}^{n})$ decay estimates for (\ref{R-E1}).
\begin{cor}\label{cor2.2} Let $U(t,x)$ be the global classical solutions of Theorem \ref{thm2.3}. Additionally, if $\widetilde{E}_{0}$ is sufficiently small,
then
\begin{eqnarray}
\|\Lambda^{k}[U(t,\cdot)-\bar{U}]\|_{L^{q}(\mathbb{R}^{n})}\lesssim \widetilde{E}_{0} (1+t)^{-\gamma_{p,q}-\frac{k}{2}}, \label{R-E21}
\end{eqnarray}
for $2\leq q\leq\infty$ and $0\leq k\leq\sigma_{c}-1-2\gamma_{2,q}$, and
\begin{eqnarray}
\|\Lambda^{k}(I-\mathcal{P})U(t,\cdot)\|_{L^{q}(\mathbb{R}^{n})}\lesssim \widetilde{E}_{0} (1+t)^{-\gamma_{p,q}-\frac{k+1}{2}}, \label{R-E22}
\end{eqnarray}
for $2\leq q\leq n$ and $0\leq k\leq\sigma_{c}-2-2\gamma_{2,q}$, where $\gamma_{p,q}=\frac{n}{2}(\frac{1}{p}-\frac{1}{q})$ is the $L^p(\mathbb{R}^{n})$-$L^q(\mathbb{R}^{n})$ decay rates for the heat kernerl.
\end{cor}

Let us point out again that our results for general dissipative systems were obtained by assuming the standard dissipative entropy in \cite{KY} and
all the time [SK] condition, where the degenerate matrix $L$ is symmetric in Definition \ref{defn2.2}. These decay results can be also applied to the 3D damped compressible Euler equations for perfect gas flow and to obtain the optimal decay rates on the framework of spatially Besov spaces, see Sect. \ref{sec:6}. Unfortunately, if the degenerate matrix $L$ is not symmetric, for example, for some concrete models, like the Timoshenko system and Euler-Maxwell system, etc., then
these decay results including Theorem \ref{thm2.3} (global-in-time existence) can no longer hold. Recently, \textsc{Ueda}, \textsc{Duan} and
the second author \cite{UDK} have formulated a new structural condition and analyze the weaker dissipative structure for general systems.  We expect that the decay framework in the present paper can be further generalized to the nonstandard but interesting case, which is the next consideration.

Finally, for readers we would like to mention the recent work \cite{SS} for the Boltzmann equation. Based on the global existence of solutions in
weighted Sobolev spaces of geometric fractional order (see \cite{GS}), \textsc{Sohinger} and \textsc{Strain} first introduced the mixed Besov space $\dot{B}^{-s}_{2,\infty}L^2_{v} (s>0)$ to investigate the optimal time decay rates. In constrast, their arguments need more regularity and can not work in spatially critical Besov spaces effectively, however, their efforts provided us the inspiration, when we deduced the decay estimates in the homogeneous Besov spaces, see Proposition \ref{prop4.2}. \\

\textbf{Notations}. Throughout the paper, we use
$(\cdot,\cdot)$ to denote the standard inner product in
the real $\mathbb{R}^{N}$ or complex $\mathbb{C}^{N}$. $f\lesssim g$ denotes $f\leq Cg$, where $C>0$
is a generic constant. $f\thickapprox g$ means $f\lesssim g$ and
$g\lesssim f$.  Denote by $\mathcal{C}([0,T],X)$ (resp., $\mathcal{C}^{1}([0,T],X)$)
the space of continuous (resp., continuously differentiable)
functions on $[0,T]$ with values in a Banach space $X$. For
simplicity, the notation $\|(f,g)\|_{X}$ means $
\|f\|_{X}+\|g\|_{X}$ with $f,g\in X$.

\section{Littlewood-Paley theory and Besov spaces}\setcounter{equation}{0}\label{sec:3}
The proofs of most of the results presented  require a
dyadic decomposition of Fourier variables, so we recall briefly the
Littlewood-Paley decomposition and related Besov spaces, see the recent book
\cite{BCD} for details.

Let us start with the Fourier transform. The Fourier transform $\hat{f}$ (or $\mathcal{F}f$)
of a $L^1$-function $f$ is given by
$$\mathcal{F}f=\int_{\mathbb{R}^{n}}f(x)e^{-2\pi x\cdot\xi}dx.$$ More
generally, the Fourier transform of a tempered distribution $f\in\mathcal{S}'$ is defined by
the dual argument in the standard way.

We now introduce a dyadic partition of $\mathbb{R}^{n}$. Choose
$\phi_{0}\in \mathcal{S}$ such that $\phi_{0}$ is even,
$$\mathrm{supp}\phi_{0}:=A_{0}=\Big\{\xi\in\mathbb{R}^{n}:\frac{3}{4}\leq|\xi|\leq\frac{8}{3}\Big\},\  \mbox{and}\ \ \phi_{0}>0\ \ \mbox{on}\ \ A_{0}.$$
Set $A_{q}=2^{q}A_{0}$ for $q\in\mathbb{Z}$. Furthermore, we define
$$\phi_{q}(\xi)=\phi_{0}(2^{-q}\xi)$$ and define $\Phi_{q}\in
\mathcal{S}$ by
$$\mathcal{F}\Phi_{q}(\xi)=\frac{\phi_{q}(\xi)}{\sum_{q\in \mathbb{Z}}\phi_{q}(\xi)}.$$
It follows that both $\mathcal{F}\Phi_{q}(\xi)$ and $\Phi_{q}$ are
even and satisfy the following properties:
$$\mathcal{F}\Phi_{q}(\xi)=\mathcal{F}\Phi_{0}(2^{-q}\xi),\ \ \ \mathrm{supp}\ \mathcal{F}\Phi_{q}(\xi)\subset A_{q},\ \ \ \Phi_{q}(x)=2^{qn}\Phi_{0}(2^{q}x)$$
and
$$\sum_{q=-\infty}^{\infty}\mathcal{F}\Phi_{q}(\xi)=\begin{cases}1,\ \ \ \mbox{if}\ \ \xi\in\mathbb{R}^{n}\setminus \{0\},
\\ 0, \ \ \ \mbox{if}\ \ \xi=0.\end{cases}
$$

Let $P$ be the class of all polynomials of $\mathbb{R}^{n}$ and denote by $S'_{0}:=S/P$ the tempered
distributions modulo polynomials. As a consequence, for any $f\in S'_{0},$ we have
$$\sum_{q=-\infty}^{\infty}\Phi_{q}\ast f=f.$$

Next, we give the definition of homogeneous Besov spaces. To do this,
 we set
$$\dot{\Delta}_{q}f=\Phi_{q}\ast f,\ \ \ \ q=0,\pm1,\pm2,...$$

\begin{defn}\label{defn3.1}
For $s\in \mathbb{R}$ and $1\leq p,r\leq\infty,$ the homogeneous
Besov spaces $\dot{B}^{s}_{p,r}$ is defined by
$$\dot{B}^{s}_{p,r}=\{f\in S'_{0}:\|f\|_{\dot{B}^{s}_{p,r}}<\infty\},$$
where
$$\|f\|_{\dot{B}^{s}_{p,r}}
=\begin{cases}\Big(\sum_{q\in\mathbb{Z}}(2^{qs}\|\dot{\Delta}_{q}f\|_{L^p})^{r}\Big)^{1/r},\
\ r<\infty, \\ \sup_{q\in\mathbb{Z}}
2^{qs}\|\dot{\Delta}_{q}f\|_{L^p},\ \ r=\infty.\end{cases}
$$\end{defn}

To define the inhomogeneous Besov spaces, we set $\Psi\in
\mathcal{C}_{0}^{\infty}(\mathbb{R}^{n})$ be even and satisfy
$$\mathcal{F}\Psi(\xi)=1-\sum_{q=0}^{\infty}\mathcal{F}\Phi_{q}(\xi).$$
It is clear that for any $f\in S'_{0}$, yields
$$\Psi*f+\sum_{q=0}^{\infty}\Phi_{q}\ast f=f.$$
We further set
$$\Delta_{q}f=\begin{cases}0,\ \ \ \ \ \ \ \, \ j\leq-2,\\
\Psi*f,\ \ \ j=-1,\cr \Phi_{q}\ast f, \ \ j=0,1,2,...,\end{cases}$$
which leads to the definition of inhomogeneous Besov spaces.

\begin{defn}\label{defn3.2}
For $s\in \mathbb{R}$ and $1\leq p,r\leq\infty,$ the inhomogeneous
Besov spaces $B^{s}_{p,r}$ is defined by
$$B^{s}_{p,r}=\{f\in S':\|f\|_{B^{s}_{p,r}}<\infty\},$$
where
$$\|f\|_{B^{s}_{p,r}}
=\begin{cases}\Big(\sum_{q=-1}^{\infty}(2^{qs}\|\Delta_{q}f\|_{L^p})^{r}\Big)^{1/r},\
\ r<\infty, \\ \sup_{q\geq-1} 2^{qs}\|\Delta_{q}f\|_{L^p},\ \
r=\infty.\end{cases}$$
\end{defn}

For convenience of reader, we present some useful facts as follows. The first one is the improved
Bernstein inequality, see, e.g., \cite{W}.

\begin{lem}\label{lem3.1}
Let $0<R_{1}<R_{2}$ and $1\leq a\leq b\leq\infty$.
\begin{itemize}
\item [(i)] If $\mathrm{Supp}\mathcal{F}f\subset \{\xi\in \mathbb{R}^{n}: |\xi|\leq
R_{1}\lambda\}$, then
\begin{eqnarray*}
\|\Lambda^{\alpha}f\|_{L^{b}}
\lesssim \lambda^{\alpha+n(\frac{1}{a}-\frac{1}{b})}\|f\|_{L^{a}}, \ \  \mbox{for any}\ \  \alpha\geq0;
\end{eqnarray*}

\item [(ii)]If $\mathrm{Supp}\mathcal{F}f\subset \{\xi\in \mathbb{R}^{n}:
R_{1}\lambda\leq|\xi|\leq R_{2}\lambda\}$, then
\begin{eqnarray*}
\|\Lambda^{\alpha}f\|_{L^{a}}\approx\lambda^{\alpha}\|f\|_{L^{a}}, \ \  \mbox{for any}\ \ \alpha\in\mathbb{R}.
\end{eqnarray*}
\end{itemize}
\end{lem}
As a consequence of the above inequality, we have
$$\|\Lambda^{\alpha} f\|_{B^s_{p, r}}\lesssim\|f\|_{B^{s +\alpha}_{p,
r}} \ (\alpha\geq0); \ \ \ \|\Lambda^{\alpha} f\|_{\dot{B}^s_{p, r}}\approx \|f\|_{\dot{B}^{s+\alpha}_{p, r}} \ (\alpha\in\mathbb{R}).$$
The following lemma gives basic embedding properties in Besov spaces.
\begin{lem}\label{lem3.2} Let $s\in \mathbb{R}$ and $1\leq
p,r\leq\infty$. Then
\begin{itemize}
\item[(1)] $\dot{B}^{0}_{p,1}\hookrightarrow L^p\hookrightarrow\dot{B}^{0}_{p,\infty},\ \  \dot{B}^{0}_{p,1}\hookrightarrow B^{0}_{p,1}$;
\item[(2)]$B^{s}_{p,r}=L^{p}\cap \dot{B}^{s}_{p,r} (s>0);$
\item[(3)]$B^{s}_{p,r}\hookrightarrow
B^{\tilde{s}}_{p,\tilde{r}}$ whenever $\tilde{s}< s$ or  $\tilde{s}=s$ and $r\leq\tilde{r}$;
\item[(4)]$\dot{B}^{s}_{p,r}\hookrightarrow \dot{B}^{s-n(\frac{1}{p}-\frac{1}{\tilde{p}})}_{\tilde{p},r}
$ and $B^{s}_{p,r}\hookrightarrow
B^{s-n(\frac{1}{p}-\frac{1}{\tilde{p}})}_{\tilde{p},r}$ whenever $p\leq\tilde{p}$;
\item[(5)]$\dot{B}^{n/p}_{p,1}\hookrightarrow\mathcal{C}_{0},\ \ B^{n/p}_{p,1}\hookrightarrow\mathcal{C}_{0}(p<\infty),$
where $\mathcal{C}_{0}$ is the space of continuous bounded functions
which decay at infinity.
\end{itemize}
\end{lem}
Let us state the Moser-type product estimates, which plays an important role in the estimate of bilinear
terms.
\begin{prop}\label{prop3.1}
Let $s>0$ and $1\leq
p,r\leq\infty$. Then $\dot{B}^{s}_{p,r}\cap L^{\infty}$ is an algebra and
$$
\|fg\|_{\dot{B}^{s}_{p,r}}\lesssim \|f\|_{L^{\infty}}\|g\|_{\dot{B}^{s}_{p,r}}+\|g\|_{L^{\infty}}\|f\|_{\dot{B}^{s}_{p,r}}.
$$
Let $s_{1},s_{2}\leq n/p$ such that $s_{1}+s_{2}>n\max\{0,\frac{2}{p}-1\}. $  Then one has
$$\|fg\|_{\dot{B}^{s_{1}+s_{2}-n/p}_{p,1}}\lesssim \|f\|_{\dot{B}^{s_{1}}_{p,1}}\|g\|_{\dot{B}^{s_{2}}_{p,1}}.$$
\end{prop}
Finally, we state a result of continuity for the composition function.
\begin{prop}\label{prop3.2}
Let $s>0$, $1\leq p, r\leq \infty$ and $F'\in
W^{[s]+1,\infty}_{loc}(I;\mathbb{R})$. Assume that $v\in \dot{B}^{s}_{p,r}\cap
L^{\infty},$ then $F(v)\in \dot{B}^{s}_{p,r}$ and
$$\|F(v)\|_{\dot{B}^{s}_{p,r}}\lesssim
(1+\|v\|_{L^{\infty}})^{n}\|F'\|_{W^{[s]+1,\infty}(I)}\|v\|_{\dot{B}^{s}_{p,r}}.$$
\end{prop}

\section{The L-P pointwise energy estimates}\setcounter{equation}{0}\label{sec:4}
In this section, we consider the perturbation system at $\bar{V}$ of the linearized dissipative system (\ref{R-E10}) satisfying the [SK] condition in Definition \ref{defn2.4}, which reads precisely:
\begin{equation}
\left\{
\begin{array}{l}
\tilde{A}^{0}\tilde{w}_{t}+\sum^{n}_{j=1}\tilde{A}^{j}\tilde{w}_{x_{j}}+L\tilde{w}=0, \\
\tilde{w}_{0}=V_{0}-\bar{V},
\end{array} \right.\label{R-E23}
\end{equation}
with $\tilde{w}=V-\bar{V}$, where
$\tilde{A}^{0}=\tilde{A}^{0}(\bar{V}),
\tilde{A}^{j}=\tilde{A}^{j}(\bar{V})$ and $L=L(\bar{V})$ are the constant
matrices.

\begin{prop}\label{prop4.1}
If $\tilde{w}_{0}\in \dot{B}^{\sigma}_{2,1}(\mathbb{R}^{n})\cap \dot{B}^{-s}_{2,\infty}(\mathbb{R}^{n})$ for $\sigma\geq0$ and $s>0$, then the solutions $\tilde{w}(t,x)$ of (\ref{R-E23}) has the decay estimate
\begin{equation}
\|\Lambda^{\ell}\tilde{w}\|_{B_{2,1}^{\sigma-\ell}}\lesssim \|\tilde{w}_{0}\|_{\dot{B}_{2,1}^{\sigma}\cap \dot{B}_{2,\infty}^{-s}}(1+t)^{-\frac{\ell+s}{2}} \label{R-E24}
\end{equation}
for $0\leq\ell\leq\sigma$. In particular, if $\tilde{w}_{0}\in \dot{B}^{\sigma}_{2,1}(\mathbb{R}^{n})\cap L^p(\mathbb{R}^{n})(1\leq p<2$), one further has
\begin{equation}
\|\Lambda^{\ell}\tilde{w}\|_{B_{2,1}^{\sigma-\ell}}\lesssim \|\tilde{w}_{0}\|_{\dot{B}_{2,1}^{\sigma}\cap L^{p}}(1+t)^{-\frac{n}{2}(\frac{1}{p}-\frac{1}{2})-\frac{\ell}{2}} \label{R-E25}
\end{equation}
for $0\leq\ell\leq\sigma$.
\end{prop}

\begin{proof}
Applying the inhomogeneous localization operator $\Delta_{q}(q\geq-1)$ to (\ref{R-E23}) gives
\begin{eqnarray}
\tilde{A}^{0}\partial_{t}\Delta_{q}\tilde{w} + \sum_{j=1}^{n}\tilde{A}^{j}\Delta_{q}\tilde{w}_{x_{j}}+L\Delta_{q}\tilde{w}=0. \label{R-E26}
\end{eqnarray}
Then we perform the Fourier transform of (\ref{R-E26}) to get
\begin{eqnarray}
\tilde{A}^{0}\partial_{t}\widehat{\Delta_{q}\tilde{w}} + (i|\xi|\tilde{A}(\omega)+L)\widehat{\Delta_{q}\tilde{w}}=0, \label{R-E27}
\end{eqnarray}
where $\omega=\xi/|\xi|$ and $\tilde{A}(\omega)=\Sigma_{j}\tilde{A}^{j}\omega_{j}.$

From Definition \ref{defn2.2} and Theorem \ref{thm2.1}, there exists a constant $c_{0}>0$ such that the following inequality holds
\begin{eqnarray}
\frac{1}{2}\frac{d}{dt}(\tilde{A}^{0}\widehat{\Delta_{q}\tilde{w}},\widehat{\Delta_{q}\tilde{w}})+c_{0}|(I-\mathcal{P})\widehat{\Delta_{q}\tilde{w}}|^2\leq 0. \label{R-E28}
\end{eqnarray}

On the other hand, multiplying (\ref{R-E27}) by $-i|\xi|\tilde{K}(\omega)$, performing the inner product with $\widehat{\Delta_{q}\tilde{w}}$ and then taking the real part of each term in the resulting equality, we obtain
\begin{eqnarray}
&&\frac{1}{2}\frac{d}{dt}\mathrm{Im}(|\xi|\tilde{K}(\omega)\tilde{A}^{0}\widehat{\Delta_{q}\tilde{w}},\widehat{\Delta_{q}\tilde{w}})
+|\xi|^2([\tilde{K}(\omega)\tilde{A}(\omega)]'\widehat{\Delta_{q}\tilde{w}},\widehat{\Delta_{q}\tilde{w}})\nonumber\\&=&|\xi|\mathrm{Im}(\tilde{K}(\omega)L\widehat{\Delta_{q}\tilde{w}},\widehat{\Delta_{q}\tilde{w}}). \label{R-E29}
\end{eqnarray}
By Theorem 2.2, it follows from $[\tilde{K}(\omega)\tilde{A}(\omega)]'+L$ is positive definite for $\omega\in \mathbb{S}^{n-1}$ that
\begin{eqnarray}
|\xi|^2([\tilde{K}(\omega)\tilde{A}(\omega)]'\widehat{\Delta_{q}\tilde{w}},\widehat{\Delta_{q}\tilde{w}})\geq c_{1}|\xi|^2|\widehat{\Delta_{q}\tilde{w}}|^2-|\xi|^2|(I-\mathcal{P})\widehat{\Delta_{q}\tilde{w}}|^2 \label{R-30}
\end{eqnarray}
for some constant $c_{1}>0$. Moreover,  by virtue of Young's inequality, the right side of (\ref{R-E29}) can be estimated as
\begin{eqnarray}
\Big||\xi|\mathrm{Im}(\tilde{K}(\omega)L\widehat{\Delta_{q}\tilde{w}},\widehat{\Delta_{q}\tilde{w}})\Big|\leq \epsilon |\xi|^2|\widehat{\Delta_{q}\tilde{w}}|^2+C_{\epsilon}|(I-\mathcal{P})\widehat{\Delta_{q}\tilde{w}}|^2 \label{R-E31}
\end{eqnarray}
for any $\epsilon>0$, where $C_{\epsilon}$ is a constant depending on $\epsilon$. Combining  (\ref{R-E29})-(\ref{R-E31}), we can deduce that
\begin{eqnarray}
&&\frac{1}{2}\frac{d}{dt}\mathrm{Im}\Big(\frac{|\xi|}{1+|\xi|^2}\tilde{K}(\omega)\tilde{A}^{0}\widehat{\Delta_{q}\tilde{w}},\widehat{\Delta_{q}\tilde{w}}\Big)
+\frac{c_{1}}{2}\frac{|\xi|^2}{1+|\xi|^2}|\widehat{\Delta_{q}\tilde{w}}|^2\nonumber\\ &\leq& C|(I-\mathcal{P})\widehat{\Delta_{q}\tilde{w}}|^2, \label{R-E32}
\end{eqnarray}
where we have chosen $\epsilon$ such that $\epsilon\leq c_{1}/2$.

Therefore, it follows from (\ref{R-E32}) and (\ref{R-E28}) that the symmetric system (\ref{R-E23}) admits a Lyapunov function of the form
\begin{eqnarray}
E[\widehat{\Delta_{q}\tilde{w}}]=\frac{1}{2}(\tilde{A}^{0}\widehat{\Delta_{q}\tilde{w}},\widehat{\Delta_{q}\tilde{w}})
+\frac{\kappa}{2}\mathrm{Im}\Big(\frac{|\xi|}{1+|\xi|^2}K(\omega)\tilde{A}^{0}\widehat{\Delta_{q}\tilde{w}},\widehat{\Delta_{q}\tilde{w}}\Big), \label{R-E33}
\end{eqnarray}
where $\kappa>0$ is a small constant. Actually, by straightforward
computations, we show that
\begin{eqnarray}
\frac{d}{dt}E[\widehat{\Delta_{q}\tilde{w}}]+(c_{0}-\kappa C)|(I-\mathcal{P})\widehat{\Delta_{q}\tilde{w}}|^2+\frac{c_{1}\kappa|\xi|^2}{2(1+|\xi|^2)}|\widehat{\Delta_{q}\tilde{w}}|^2\leq0, \label{R-E34}
\end{eqnarray}
where we choose $\kappa>0$ so small that $c_{0}-\kappa C\geq0$ and $E[\widehat{\Delta_{q}\tilde{w}}]\approx|\widehat{\Delta_{q}\tilde{w}}|^2$, since $\tilde{A}^{0}$ is positive definite. To get the desired decay estimates, in the following, we divide (\ref{R-E34}) into high-frequency and low-frequency cases, respectively.

\underline{\textit{Case 1}($q\geq0$)}

In this case, $|\xi|\sim 2^{q}$, there exists a constant $c_{2}>0$ such that the differential inequality holds
\begin{eqnarray}
\frac{d}{dt}E[\widehat{\Delta_{q}\tilde{w}}]+c_{2}|\widehat{\Delta_{q}\tilde{w}}|^2\leq0, \label{R-E35}
\end{eqnarray}
which implies that
\begin{eqnarray}
\|\Delta_{q}\tilde{w}\|_{L^2}\lesssim e^{-c_{2}t}\|\Delta_{q}\tilde{w}_{0}\|_{L^2}. \label{R-E36}
\end{eqnarray}
Then multiplying the factor $2^{q\sigma}$  on both sides of (\ref{R-E36}) and summing the resulting inequality over $q\geq0$, we arrive at
\begin{eqnarray}
&&\sum_{q\geq0}2^{q(\sigma-\ell)}\|\Delta_{q}\Lambda^{\ell}\tilde{w}\|_{L^2}\nonumber\\&\lesssim& e^{-c_{2}t}\sum_{q\geq0}2^{q(\sigma-\ell)}\|\Delta_{q}\Lambda^{\ell}\tilde{w}_{0}\|_{L^2}\lesssim e^{-c_{2}t} \|\tilde{w}_{0}\|_{\dot{B}_{2,1}^{\sigma}}, \label{R-E37}
\end{eqnarray}
where we have used Lemma \ref{lem3.1}.

\underline{\textit{Case 2}($q=-1$)}

In this case, there exists a constant $c_{3}>0$ such that the differential inequality holds
\begin{eqnarray}
\frac{d}{dt}E[\widehat{\tilde{w}_{-1}}]+c_{3}|\xi|^2|\widehat{\tilde{w}_{-1}}|^2\leq0, \label{R-E38}
\end{eqnarray}
where $\tilde{w}_{-1}:=\Delta_{-1}\tilde{w}.$

Multiplying (\ref{R-E38}) with $|\xi|^{2\ell}$ and integrating the resulting inequality  over $\mathbb{R}^{n}_{\xi}$, with the aid of Plancherel's theorem, we arrive at
\begin{eqnarray}
\frac{d}{dt}\mathcal{E}[\tilde{w}_{-1}]^2+c_{3}\|\Lambda^{\ell+1}\tilde{w}_{-1}\|^2_{L^2}\leq0, \label{R-E39}
\end{eqnarray}
where $$\mathcal{E}[\tilde{w}_{-1}]:=\Big(\int_{\mathbb{R}^{n}_{\xi}}|\xi|^{2\ell}E[\widehat{\tilde{w}_{-1}}]d\xi\Big)^{1/2}\approx\|\Lambda^{\ell}\tilde{w}_{-1}\|_{L^2}.$$
According to the interpolation inequality related the Besov space $\dot{B}^{-s}_{2,\infty}$ (see Lemma \ref{lem8.2}), we arrive at
\begin{eqnarray}
\|\Lambda^{\ell}\tilde{w}_{-1}\|_{L^2} &\lesssim&\|\Lambda^{\ell+1}\tilde{w}_{-1}\|^{\theta}_{L^2}\|\tilde{w}_{-1}\|^{1-\theta}_{\dot{B}^{-s}_{2,\infty}}\ \ \  \Big(\theta=\frac{\ell+s}{\ell+1+s}\Big)\nonumber\\ &\lesssim& \|\Lambda^{\ell+1}\tilde{w}_{-1}\|^{\theta}_{L^2}\|\tilde{w}\|^{1-\theta}_{\dot{B}^{-s}_{2,\infty}}, \label{R-E40}
\end{eqnarray}

In addition, by applying the homogeneous operator $\dot{\Delta}_{q}(q\in \mathbb{Z})$ to the system (\ref{R-E23}) and performing the inter product with
$\dot{\Delta}_{q}\tilde{w}$, we can infer that there exists a constant $c_{4}>0$ such that
\begin{eqnarray}
\frac{1}{2}\frac{d}{dt}(\tilde{A}^{0}\widehat{\dot{\Delta}_{q}\tilde{w}},\widehat{\dot{\Delta}_{q}\tilde{w}})+c_{4}\|(I-\mathcal{P})\dot{\Delta}_{q}\tilde{w}\|^2_{L^2}\leq0, \label{R-E41}
\end{eqnarray}
which immediately leads to
\begin{eqnarray}
\|\tilde{w}\|_{\dot{B}^{-s}_{2,\infty}}\leq\|\tilde{w}_{0}\|_{\dot{B}^{-s}_{2,\infty}}.  \label{R-E42}
\end{eqnarray}
Together with (\ref{R-E40}) and (\ref{R-E42}), we are led to the differential inequality
\begin{eqnarray}
\frac{d}{dt}\mathcal{E}[\tilde{w}_{-1}]^2+C\|\tilde{w}_{0}\|_{\dot{B}^{-s}_{2,\infty}}^{-\frac{2}{s}}(\|\Lambda^{\ell}\tilde{w}_{-1}\|^2_{L^2})^{1+\frac{1}{\ell+s}}\leq0,\label{R-E43}
\end{eqnarray}
which yields
\begin{eqnarray}
\|\Lambda^{\ell}w_{-1}\|_{L^2}\lesssim\|\tilde{w}_{0}\|_{\dot{B}^{-s}_{2,\infty}}(1+t)^{-\frac{\ell+s}{2}}.\label{R-E44}
\end{eqnarray}
Hence, it follows from (\ref{R-E37}) and (\ref{R-E44}) that
\begin{eqnarray}
&&\|\Lambda^{\ell}\tilde{w}\|_{B_{2,1}^{\sigma-\ell}}\nonumber\\&\lesssim& \|\tilde{w}_{0}\|_{\dot{B}^{-s}_{2,\infty}}(1+t)^{-\frac{\ell+s}{2}}+ \|\tilde{w}_{0}\|_{\dot{B}_{2,1}^{\sigma}}e^{-c_{2}t}\nonumber\\&\lesssim& \|\tilde{w}_{0}\|_{\dot{B}_{2,1}^{\sigma}\cap \dot{B}_{2,\infty}^{-s}}(1+t)^{-\frac{\ell+s}{2}}. \label{R-E45}
\end{eqnarray}

Based on (\ref{R-E45}), the optimal decay estimate (\ref{R-E25}) follows from the embedding $L^p(\mathbb{R}^{n})\hookrightarrow\dot{B}^{-s}_{2,\infty}(\mathbb{R}^{n})(s=n(1/p-1/2))$ in Lemma \ref{lem8.5}. Therefore, the proof of Proposition \ref{prop4.1} is complete.
\end{proof}


Additionally, if the inhomogeneous operator $\Delta_{q}(q\geq-1)$ is replaced with the homogeneous operator $\dot{\Delta}_{q}(q\in \mathbb{Z})$ in the proof of Proposition \ref{prop4.1}, we
can obtain the decay estimates in homogeneous Besov spaces by using the low-frequency and high-frequency decomposition techniques.
\begin{prop}\label{prop4.2}
If $\tilde{w}_{0}\in \dot{B}^{\sigma}_{2,1}(\mathbb{R}^{n})\cap \dot{B}^{-s}_{2,\infty}(\mathbb{R}^{n})$ for $\sigma\in \mathbb{R}, s\in \mathbb{R}$ satisfying $\sigma+s>0$, then the solution $\tilde{w}(t,x)$ of (\ref{R-E23}) has the decay estimate
\begin{equation}
\|\tilde{w}\|_{\dot{B}_{2,1}^{\sigma}}\lesssim \|\tilde{w}_{0}\|_{\dot{B}_{2,1}^{\sigma}\cap \dot{B}_{2,\infty}^{-s}}(1+t)^{-\frac{\sigma+s}{2}}. \label{R-E244}
\end{equation}
In particular, if $\tilde{w}_{0}\in \dot{B}^{\sigma}_{2,1}(\mathbb{R}^{n})\cap L^p(\mathbb{R}^{n})(1\leq p<2$), one further has
\begin{equation}
\|\tilde{w}\|_{\dot{B}_{2,1}^{\sigma}}\lesssim \|\tilde{w}_{0}\|_{\dot{B}_{2,1}^{\sigma}\cap L^{p}}(1+t)^{-\frac{n}{2}(\frac{1}{p}-\frac{1}{2})-\frac{\sigma}{2}}. \label{R-E255}
\end{equation}
\end{prop}
\begin{proof}
Due to the fact that the operator $\dot{\Delta}_{q}$ consists with $\Delta_{q}$ for $q\geq0$, the same process leading to (\ref{R-E37}) gives the high-frequency inequality:
\begin{eqnarray}
\sum_{q\geq0}2^{q\sigma}\|\tilde{w}_{q}\|_{L^2}\lesssim e^{-c_{2}t}\sum_{q\geq0}2^{q\sigma}\|\tilde{w}_{q0}\|_{L^2}\lesssim e^{-c_{2}t} \|\tilde{w}_{0}\|_{\dot{B}_{2,1}^{\sigma}}, \label{R-E377}
\end{eqnarray}
where $\tilde{w}_{q}:=\dot{\Delta}_{q}\tilde{w}.$ However, the low-frequency estimate will be shown by a different idea inspired by \cite{SS}, since Lemma \ref{lem8.2} is not true for all $\sigma, s\in \mathbb{R}$. More precisely, we first have an analogue of the inequality (\ref{R-E38}):
\begin{eqnarray}
\frac{d}{dt}E[\widehat{\tilde{w}_{q}}]+c_{3}|\xi|^2|\widehat{\tilde{w}_{q}}|^2\leq0, \label{R-E388}
\end{eqnarray}
where $|\xi|\sim 2^{q}(q\leq0)$. Then it follows from (\ref{R-E388}) that
\begin{eqnarray}
E[\widehat{\tilde{w}_{q}}]\leq e^{-c_{3}2^{2q}t}E[\widehat{\tilde{w}_{q0}}]. \label{R-E389}
\end{eqnarray}
By integrating (\ref{R-E389}) over $\mathbb{R}^{n}_{\xi}$, with the aid of Plancherel's theorem, we arrive at
\begin{eqnarray}
\|\tilde{w}_{q}\|_{L^2}\lesssim e^{-\frac{1}{2}c_{3}(2^{q}\sqrt{t})^2}\|\tilde{w}_{q0}\|_{L^2}, \label{R-E446}
\end{eqnarray}
that is,
\begin{eqnarray}
2^{q\sigma}\|\tilde{w}_{q}\|_{L^2}\lesssim \|\tilde{w}_{0}\|_{\dot{B}^{-s}_{2,\infty}}(1+t)^{-\frac{\sigma+s}{2}}
\Big[(2^{q}\sqrt{t})^{\sigma+s}e^{-\frac{1}{2}c_{3}(2^{q}\sqrt{t})^2}\Big],\label{R-E447}
\end{eqnarray}
which immediately implies that
\begin{eqnarray}
\sum_{q<0}2^{q\sigma}\|\tilde{w}_{q}\|_{L^2}\lesssim\|\tilde{w}_{0}\|_{\dot{B}^{-s}_{2,\infty}}(1+t)^{-\frac{\sigma+s}{2}}.\label{R-E448}
\end{eqnarray}
Therefore, we conclude that
\begin{eqnarray}
\|\tilde{w}\|_{\dot{B}_{2,1}^{\sigma}}&=&\sum_{q<0}2^{q\sigma}\|w_{q}\|_{L^2}+\sum_{q\geq0}2^{q\sigma}\|w_{q}\|_{L^2}
\nonumber\\&\lesssim& \|\tilde{w}_{0}\|_{\dot{B}^{-s}_{2,\infty}}(1+t)^{-\frac{\sigma+s}{2}}+e^{-c_{2}t} \|\tilde{w}_{0}\|_{\dot{B}_{2,1}^{\sigma}}
\nonumber\\&\lesssim& \|\tilde{w}_{0}\|_{\dot{B}_{2,1}^{\sigma}\cap \dot{B}_{2,\infty}^{-s}}(1+t)^{-\frac{\sigma+s}{2}}
\end{eqnarray}
for $\sigma+s>0$, which deduces (\ref{R-E244}), and (\ref{R-E255}) is followed by the $L^p$
embedding. Hence, the proof of Proposition \ref{prop4.2} is finished.
\end{proof}

\section{Decay estimates for symmetric hyperbolic systems}\setcounter{equation}{0}\label{sec:5}
The aim of this section is to deduce decay estimates for the nonlinear
symmetric hyperbolic system with partial source terms. Our main idea is similar
to \cite{Ka}, where the decay rate of solutions could be obtained by using the Duhamel principle and the time-weighted energy approach.
However, due to the Littlewood-Paley pointwise energy estimates, first of all, we develop the frequency-localization
Duhamel principle, then we introduce the time-weighted energy functionals containing the decay rates, to
obtain the decay estimates of solutions for the corresponding nonlinear problem. A new ingredient is that the low-frequency and high-frequency
decomposition methods and the improved Gagliardo-Nirenberg-Sobolev inequality are mainly used.

To do this, (\ref{R-E7}) can be written as the following form
at $\bar{V}$:
\begin{eqnarray}
\tilde{A}^{0}\tilde{w}_{t}+\sum^{n}_{j=1}\tilde{A}^{j}\tilde{w}_{x_{j}}+L\tilde{w}=\mathcal{R},\label{R-E46}
\end{eqnarray}
with $\tilde{w}=V-\bar{V}$, where
$\tilde{A}^{0}=\tilde{A}^{0}(\bar{V}),
\tilde{A}^{j}=\tilde{A}^{j}(\bar{V})$ and $L=L(\bar{V})$ are the constant
matrices, and $\mathcal{R}=\mathcal{R}_{1}+\mathcal{R}_{2}$ with
$$\mathcal{R}_{1}:=-\sum^{n}_{j=1}\tilde{A}^{0}\Big(\tilde{A}^{0}(V)^{-1}\tilde{A}^{j}(V)-(\tilde{A}^{0})^{-1}\tilde{A}^{j}\Big)V_{x_{j}},$$
$$\mathcal{R}_{2}:=\tilde{A}^{0}\Big\{-\Big(\tilde{A}^{0}(V)^{-1}-(\tilde{A}^{0})^{-1}\Big)LV+\tilde{A}^{0}(V)^{-1}\tilde{r}(V)\Big\}.$$
The initial condition is supplemented by
\begin{equation}
\tilde{w}_{0}=V_{0}-\bar{V}. \label{R-E47}
\end{equation}

Next, we denote by $\mathcal{G}(t,x)$ the Green matrix associated with the linear Cauchy problem (\ref{R-E23}):
\begin{eqnarray}
\widehat{\mathcal{G}f}(t,\xi)=e^{t\Phi(i\xi)}\hat{f}(\xi), \label{R-E48}
\end{eqnarray}
where $$\Phi(i\xi)=-(A^{0})^{-1}[A(i\xi)+L]$$ with
$A(i\xi)=i\sum_{j=1}^{n}\tilde{A}^{j}\xi_{j}$. Then $\mathcal{G}(t,x)\tilde{w}_{0}$ is the solution of (\ref{R-E23}).
 Furthermore, in terms of the Green matrix $\mathcal{G}(t,x)$, the solution of (\ref{R-E46})-(\ref{R-E47}) can be expressed as by the standard Duhamel principle
\begin{eqnarray}
\tilde{w}(t,x)=\mathcal{G}(t,x)\tilde{w}_{0}+\int^{t}_{0}\mathcal{G}(t-\tau, x)\mathcal{R}(\tau)d\tau. \label{R-E49}
\end{eqnarray}

In addition, we also have the frequency-localization Duhamel principle.
\begin{lem}\label{lem5.1}
Suppose that $\tilde{w}=V(t,x)-\bar{V}$ is a solution of (\ref{R-E46})-(\ref{R-E47}). Then it holds that
\begin{eqnarray}
\Delta_{q}\Lambda^{\ell}\tilde{w}(t,x)=\Delta_{q}\Lambda^{\ell}[\mathcal{G}(t,x)\tilde{w}_{0}]
+\int^{t}_{0}\Delta_{q}\Lambda^{\ell}[\mathcal{G}(t-\tau,x)\mathcal{R}(\tau)]d\tau\label{R-E50}
\end{eqnarray}
for $q\geq-1$ and $\ell\in \mathbb{R}$, and
\begin{eqnarray}
\dot{\Delta}_{q}\Lambda^{\ell}\tilde{w}(t,x)=\dot{\Delta}_{q}\Lambda^{\ell}[\mathcal{G}(t,x)\tilde{w}_{0}]
+\int^{t}_{0}\dot{\Delta}_{q}\Lambda^{\ell}[\mathcal{G}(t-\tau,x)\mathcal{R}(\tau)]d\tau\label{R-E500}
\end{eqnarray}
for $q\in\mathbb{Z}$ and $\ell\in \mathbb{R}$.
\end{lem}
\begin{proof} We only show (\ref{R-E50}) and (\ref{R-E500}) follows similarly.
Applying the operator $\Delta_{q}\Lambda^{\ell}$ to (\ref{R-E46})-(\ref{R-E47}) gives
\begin{equation}
\left\{
\begin{array}{l}
\tilde{A}^{0}\Delta_{q}\Lambda^{\ell}\tilde{w}_{t}+\sum^{n}_{j=1}\tilde{A}^{j}\Delta_{q}\Lambda^{\ell}\tilde{w}_{x_{j}}+L\Delta_{q}\Lambda^{\ell}\tilde{w}
=\Delta_{q}\Lambda^{\ell}\mathcal{R},\\
\Delta_{q}\Lambda^{\ell}\tilde{w}|_{t=0}=\Delta_{q}\Lambda^{\ell}\tilde{w}_{0}.
\end{array} \right. \label{R-E51}
\end{equation}
Note that the definition of $\mathcal{G}$ in (\ref{R-E48}), the equality (\ref{R-E50}) follows from the
frequency-localization system (\ref{R-E51}) and (\ref{R-E49}) immediately.
\end{proof}

In what follows, our task is to prove the decay estimates for nonlinear problem (\ref{R-E46})-(\ref{R-E47}). For this purpose, we introduce some time-weighted sup-norms:
$$
\mathcal{E}_{0}(t):=\sup_{0\leq\tau\leq t}\|\tilde{w}(\tau)\|_{B^{\sigma_{c}}_{2,1}};
$$
\begin{eqnarray*}
\mathcal{E}_{1}(t)&:=&\sup_{0\leq\ell<(\sigma_{c}-1)}\sup_{0\leq\tau\leq t}(1+\tau)^{\frac{s+\ell}{2}}\|\Lambda^{\ell}\tilde{w}(\tau)\|_{B^{\sigma_{c}-1-\ell}_{2,1}}\nonumber\\&&+\sup_{0\leq\tau\leq t}(1+\tau)^{\frac{s+\sigma_{c}-1}{2}}\|\Lambda^{\sigma_{c}-1}\tilde{w}(\tau)\|_{\dot{B}^{0}_{2,1}};
\end{eqnarray*}
\begin{eqnarray*}
\mathcal{E}_{2}(t)&:=&\sup_{0\leq\ell<(\sigma_{c}-2)}\sup_{0\leq\tau\leq t}(1+\tau)^{\frac{s+\ell+1}{2}}\|\Lambda^{\ell}(I-\mathcal{P})\tilde{w}(\tau)\|_{B^{\sigma_{c}-2-\ell}_{2,1}}
\nonumber\\&&+\sup_{0\leq\tau\leq t}(1+\tau)^{\frac{s+\sigma_{c}-1}{2}}\|\Lambda^{\sigma_{c}-2}(I-\mathcal{P})\tilde{w}(\tau)\|_{\dot{B}^{0}_{2,1}}
\end{eqnarray*}
and further set
$$\mathcal{E}(t):=\mathcal{E}_{1}(t)+\mathcal{E}_{2}(t).$$

Here let us note again that the time-weighted sup-norms are different in regard to the derivative index, since we take care of the topological relation between homogeneous Besov spaces and inhomogeneous Besov spaces, which can be regarded as the great improvement of those in \cite{KY,Ma}. Precisely, we shall prove the following result.
\begin{prop}\label{prop5.1}
Let $\tilde{w}=V(t,x)-\bar{V}$ be the global classical solution in the sense of Theorem \ref{thm2.3}. Suppose that $\tilde{w}_{0}-\bar{w}\in B^{\sigma_{c}}_{2,1}\cap \dot{B}^{-s}_{2,\infty}(0<s\leq n/2)$ and the norm
$E_{0}:=\|w_{0}-\bar{w}\|_{B^{\sigma_{c}}_{2,1}\cap \dot{B}^{-s}_{2,\infty}}$ is sufficiently small. Then it holds that
\begin{eqnarray}
\|\Lambda^{\ell}\tilde{w}(t)\|_{X_{1}}\lesssim E_{0}(1+t)^{-\frac{s+\ell}{2}} \label{R-E52}
\end{eqnarray}
for $0\leq\ell\leq \sigma_{c}-1$, where $X_{1}:=B^{\sigma_{c}-1-\ell}_{2,1}$ if $0\leq\ell<\sigma_{c}-1$ and $X_{1}:=\dot{B}^{0}_{2,1}$ if $\ell=\sigma_{c}-1$;
\begin{eqnarray}
\|\Lambda^{\ell}(I-\mathcal{P})\tilde{w}(t)\|_{X_{2}}\lesssim E_{0}(1+t)^{-\frac{s+\ell+1}{2}} \label{R-E53}
\end{eqnarray}
for $0\leq\ell\leq \sigma_{c}-2$, where $X_{2}:=B^{\sigma_{c}-2-\ell}_{2,1}$ if $0\leq\ell<\sigma_{c}-2$ and $X_{2}:=\dot{B}^{0}_{2,1}$ if $\ell=\sigma_{c}-2$.
\end{prop}

Proposition \ref{prop5.1} mainly depends on an energy inequality related to those time-weighted quantities, which is included in the following proposition.
\begin{prop}\label{prop5.2}
Let $\tilde{w}=V(t,x)-\bar{V}$ be the global classical solution in the sense of Theorem \ref{thm2.3}. Additional, if $\tilde{w}_{0}-\bar{w}\in \dot{B}^{-s}_{2,\infty}(0<s\leq n/2)$, then
\begin{eqnarray}
\mathcal{E}(t)\lesssim E_{0}+\mathcal{E}^{2}(t)+\mathcal{E}_{0}(t)\mathcal{E}(t), \label{R-E54}
\end{eqnarray}
where $E_{0}$ is defined as Proposition 5.1.
\end{prop}

Here we plan to divide the proof of Proposition 5.2 into several lemmas for clarity, since it is a little long.
The first result is the nonlinear low-frequency estimate for solutions.

\begin{lem}\label{lem5.2}
Under the assumption of Proposition \ref{prop5.2}, we have
\begin{eqnarray}
\|\Delta_{-1}\Lambda^{\ell}\tilde{w}\|_{L^2}\lesssim \|\tilde{w}_{0}\|_{\dot{B}^{-s}_{2,\infty}} (1+t)^{-\frac{s+\ell}{2}}+(1+t)^{-\frac{s+\ell}{2}}\mathcal{E}^{2}(t). \label{R-E55}
\end{eqnarray}
for $0\leq\ell<\sigma_{c}-1$, and
\begin{eqnarray}
\sum_{q<0}\|\dot{\Delta}_{q}\Lambda^{\sigma_{c}-1}\tilde{w}\|_{L^2}\lesssim \|\tilde{w}_{0}\|_{\dot{B}^{-s}_{2,\infty}} (1+t)^{-\frac{s+\sigma_{c}-1}{2}}+(1+t)^{-\frac{s+\sigma_{c}-1}{2}}\mathcal{E}^{2}(t). \label{R-E555}
\end{eqnarray}
\end{lem}

\begin{proof} We first prove (\ref{R-E55}).
From (\ref{R-E44}), we have
\begin{eqnarray}
\|\Delta_{-1}\Lambda^{\ell}[\mathcal{G}(x,t)\tilde{w}_{0}]\|_{L^2} \lesssim \|\tilde{w}_{0}\|_{\dot{B}^{-s}_{2,\infty}} (1+t)^{-\frac{s+\ell}{2}}. \label{R-E56}
\end{eqnarray}
On the other hand, it follows from (\ref{R-E50}) and (\ref{R-E56}) that
\begin{eqnarray}
&&\|\Delta_{-1}\Lambda^{\ell}\tilde{w}\|_{L^2}
\nonumber\\&\leq&\|\Delta_{-1}\Lambda^{\ell}[\mathcal{G}(x,t)\tilde{w}_{0}]\|_{L^2}+\int^{t}_{0}\|\Delta_{-1}\Lambda^{\ell}[\mathcal{G}(x,t-\tau)\mathcal{R}(\tau)]\|_{L^2}d\tau.
\nonumber\\&\lesssim &  \|\tilde{w}_{0}\|_{\dot{B}^{-s}_{2,\infty}} (1+t)^{-\frac{s+\ell}{2}}+
\int^{t}_{0}(1+t-\tau)^{-\frac{s+\ell}{2}}\|\mathcal{R}(\tau)\|_{\dot{B}^{-s}_{2,\infty}}d\tau
\nonumber\\&\lesssim &  \|\tilde{w}_{0}\|_{\dot{B}^{-s}_{2,\infty}} (1+t)^{-\frac{s+\ell}{2}}+
\int^{t}_{0}(1+t-\tau)^{-\frac{s+\ell}{2}}\|\mathcal{R}(\tau)\|_{L^p}d\tau, \label{R-E57}
\end{eqnarray}
for $1/p=s/n+1/2$, where we used Lemma \ref{lem8.5} in the last step. Our next object is to estimate the norm
$\|\mathcal{R}(\tau)\|_{L^p}$ for partial source terms.

Note that $\|\mathcal{R}(\tau)\|_{L^p}\leq\|\mathcal{R}_{1}(\tau)\|_{L^p}+\|\mathcal{R}_{2}(\tau)\|_{L^p}$, so we first deal with the norm
$\|\mathcal{R}_{1}(\tau)\|_{L^p}$. To do this, we need to develop the Gagliardo-Nirenberg-Sobolev inequality as in \cite{N}, which allows to the fractional derivatives, see Lemma \ref{lem8.4}. In the following, we estimate the norm
$\|\mathcal{R}_{1}(\tau)\|_{L^p}$  by using different interpolation inequalities.

\underline{\textit{Case 1} ($0<s\leq n/2-1$)}
Set $\mathcal{H}_{1}(\tilde{w}):=\tilde{A}^{0}(V)^{-1}\tilde{A}^{j}(V)-(\tilde{A}^{0})^{-1}\tilde{A}^{j}$. It follows from the H\"{o}lder's inequality and Taylor's formula of first order that
\begin{eqnarray}
\|\mathcal{R}_{1}(\tau)\|_{L^p}&\lesssim &\|\mathcal{H}_{1}(\tilde{w})\|_{L^{n/s}}\|\nabla\tilde{w}\|_{L^2}
\nonumber\\&
\lesssim & \|\tilde{w}\|_{L^{n/s}}\|\nabla\tilde{w}\|_{L^2},
\end{eqnarray}
since the global classical solution $V(t,x)$ takes values in a neighborhood of
$\bar{V}\in\mathcal{M}$. Applying Lemma \ref{lem8.4} (taking $r=2$) and Young's inequality to get
\begin{eqnarray}
\|\mathcal{R}_{1}(\tau)\|_{L^p}&\lesssim &\|\Lambda\tilde{w}\|^{\theta}_{L^2}\|\Lambda^{\alpha}\tilde{w}\|^{1-\theta}_{L^2}\|\Lambda\tilde{w}\|_{L^2}
\nonumber\\&
\lesssim & (\|\Lambda\tilde{w}\|_{L^2}+\|\Lambda^{\alpha}\tilde{w}\|_{L^2})\|\Lambda\tilde{w}\|_{L^2}, \label{R-E58}
\end{eqnarray}
where $n/2-s<\alpha\leq \sigma_{c}-1$ and $\theta=\frac{\alpha+s-n/2}{\alpha-1}$. Recall the definition of $\mathcal{E}_{1}(t)$,
we need to estimate the term of the right side of (\ref{R-E58}) at different manners.

In case that $n/2-s<\alpha<\sigma_{c}-1$, it follows from (\ref{R-E58}) that
\begin{eqnarray}
\|\mathcal{R}_{1}(\tau)\|_{L^p}&\lesssim &(\|\Lambda\tilde{w}\|_{B^{\sigma_{c}-2}_{2,1}}
+\|\Lambda^{\alpha}\tilde{w}\|_{B^{\sigma_{c}-1-\alpha}_{2,1}})\|\Lambda\tilde{w}\|_{B^{\sigma_{c}-2}_{2,1}}
\nonumber\\&
\lesssim &\Big[(1+\tau)^{-\frac{s}{2}-\frac{1}{2}}+(1+\tau)^{-\frac{s}{2}-\frac{\alpha}{2}}\Big](1+\tau)^{-\frac{s}{2}-\frac{1}{2}}\mathcal{E}^2_{1}(t)
\nonumber\\&
\lesssim &(1+\tau)^{-s-1}\mathcal{E}^2_{1}(t). \label{R-E59}
\end{eqnarray}
In case that $\alpha=\sigma_{c}-1$, it follows from (\ref{R-E58}) that
\begin{eqnarray}
\|\mathcal{R}_{1}(\tau)\|_{L^p}&\lesssim &(\|\Lambda\tilde{w}\|_{B^{\sigma_{c}-2}_{2,1}}+\|\Lambda^{\sigma_{c}-1}\tilde{w}\|_{\dot{B}^{0}_{2,1}})\|\Lambda\tilde{w}\|_{B^{\sigma_{c}-2}_{2,1}}
\nonumber\\&
\lesssim &\Big[(1+\tau)^{-\frac{s}{2}-\frac{1}{2}}+(1+\tau)^{-\frac{s}{2}-\frac{\sigma_{c}-1}{2}}\Big](1+\tau)^{-\frac{s}{2}-\frac{1}{2}}\mathcal{E}^2_{1}(t)
\nonumber\\&
\lesssim &(1+\tau)^{-s-1}\mathcal{E}^2_{1}(t)\label{R-E599}
\end{eqnarray}
Here, let us point out $s+1>1$, which is crucial in the sequent integral with respect to $t$.

\underline{\textit{Case 2} ($n/2-1<s\leq n/2$)}  It follows from the H\"{o}lder's inequality,
Lemma \ref{lem8.4} and Young's inequality that
\begin{eqnarray}
\|\mathcal{R}_{1}(\tau)\|_{L^p}&\lesssim & \|\tilde{w}\|_{L^{n/s}}\|\nabla\tilde{w}\|_{L^2}
\nonumber\\&
\lesssim &\|\tilde{w}\|^{\theta}_{L^{2}}\|\Lambda \tilde{w}\|^{1-\theta}_{L^2}\|\Lambda\tilde{w}\|_{L^2}
\nonumber\\&
\lesssim &(\|\tilde{w}\|_{L^{2}}+\|\Lambda\tilde{w}\|_{L^2})\|\Lambda\tilde{w}\|_{L^2}, \label{R-E60}
\end{eqnarray}
where $\theta=1+s-n/2$.  Recall the definition of $\mathcal{E}_{1}(t)$, we arrive at
\begin{eqnarray}
\|\mathcal{R}_{1}(\tau)\|_{L^p}&\lesssim & (\|\tilde{w}\|_{B^{\sigma_{c}-1}_{2,1}}+\|\Lambda\tilde{w}\|_{B^{\sigma_{c}-2}_{2,1}})\|\Lambda\tilde{w}\|_{B^{\sigma_{c}-2}_{2,1}}
\nonumber\\&
\lesssim &\Big[(1+\tau)^{-\frac{s}{2}}+(1+\tau)^{-\frac{s}{2}-\frac{1}{2}}\Big](1+\tau)^{-\frac{s}{2}-\frac{1}{2}}\mathcal{E}^2_{1}(t)
\nonumber\\&
\lesssim& (1+\tau)^{-s-\frac{1}{2}}\mathcal{E}^2_{1}(t), \label{R-E61}
\end{eqnarray}
where $s+1/2>(n-1)/2\geq1$.

Next, we bound the norm $\|\mathcal{R}_{2}(\tau)\|_{L^p}$. For $\mathcal{R}_{2}$, we write $\mathcal{R}_{2}:=\mathcal{R}_{21}+\mathcal{R}_{22}$. Set
 $\mathcal{H}_{2}(\tilde{w}):=-(\tilde{A}^{0}(V)^{-1}-(\tilde{A}^{0})^{-1})$. Then
\begin{eqnarray}
 \|\mathcal{R}_{21}\|_{L^p}\lesssim\|\mathcal{H}_{2}(\tilde{w})\|_{L^{n/s}}\|(I-\mathcal{P})\tilde{w}\|_{L^2}. \label{R-E62}
\end{eqnarray}
Thanks to the time-weighted quantities $\mathcal{E}_{1}(t)$ and $\mathcal{E}_{2}(t)$, it follows
from the similar analysis as (\ref{R-E58})-(\ref{R-E61}) that
\begin{eqnarray}
 \|\mathcal{R}_{21}\|_{L^p}\lesssim
 \begin{cases}
 (1+\tau)^{-s-1}\mathcal{E}_{1}(t)\mathcal{E}_{2}(t),\ \ \ \ 0<s\leq n/2-1;\\
 (1+\tau)^{-s-\frac{1}{2}}\mathcal{E}_{1}(t)\mathcal{E}_{2}(t),\ \ \ n/2-1<s\leq n/2.
 \end{cases} \label{R-E63}
\end{eqnarray}
Recall the useful inequality (\ref{R-E9}), we have
\begin{eqnarray}
 \|\mathcal{R}_{22}\|_{L^p}\lesssim\|\tilde{r}(V)\|_{L^p}\lesssim\|\tilde{w}\|_{L^{n/s}}\|(I-\mathcal{P})\tilde{w}\|_{L^2}. \label{R-E64}
\end{eqnarray}
Furthmore, in a similar way, we can also arrive at
\begin{eqnarray}
 \|\mathcal{R}_{22}\|_{L^p}\lesssim
 \begin{cases}
 (1+\tau)^{-s-1}\mathcal{E}_{1}(t)\mathcal{E}_{2}(t),\ \ \ \ 0<s\leq n/2-1;\\
 (1+\tau)^{-s-\frac{1}{2}}\mathcal{E}_{1}(t)\mathcal{E}_{2}(t),\ \ \ n/2-1<s\leq n/2.
 \end{cases}\label{R-E65}
\end{eqnarray}
Together with inequalities (\ref{R-E59}), (\ref{R-E599}), (\ref{R-E61}), (\ref{R-E63}) and (\ref{R-E65}), we conclude that
\begin{eqnarray}
\|\mathcal{R}(\tau)\|_{L^p}\lesssim
 \begin{cases}
 (1+\tau)^{-s-1}\mathcal{E}^2(t),\ \ \ \ 0<s\leq n/2-1;\\
 (1+\tau)^{-s-\frac{1}{2}}\mathcal{E}^2(t),\ \ \ n/2-1<s\leq n/2.
 \end{cases} \label{R-E66}
\end{eqnarray}
Therefore, we have
\begin{eqnarray}
&&\int^{t}_{0}(1+t-\tau)^{-\frac{s+\ell}{2}}\|\mathcal{R}(\tau)\|_{L^p}d\tau\nonumber\\&
\lesssim &  \begin{cases}
 \int^{t}_{0}(1+t-\tau)^{-\frac{s+\ell}{2}}(1+\tau)^{-s-1}\mathcal{E}^2(t)d\tau,\ \ \ \ 0<s\leq n/2-1;\\
 \int^{t}_{0}(1+t-\tau)^{-\frac{s+\ell}{2}}(1+\tau)^{-s-\frac{1}{2}}\mathcal{E}^2(t)d\tau,\ \ \ n/2-1<s\leq n/2;
 \end{cases}
 \nonumber\\&
\lesssim &
 (1+t)^{-\frac{s+\ell}{2}}\mathcal{E}^2(t). \label{R-E67}
\end{eqnarray}
Noticing that (\ref{R-E57}) and (\ref{R-E67}), the desired inequality (\ref{R-E55}) is followed immediately.\\

Concerning $\ell=\sigma_{c}-1$, from (\ref{R-E448}), we have
\begin{eqnarray}
\sum_{q<0}\|\dot{\Delta}_{q}\Lambda^{\sigma_{c}-1}[\mathcal{G}(t,x)\tilde{w}_{0}]\|_{L^2} \lesssim \|\tilde{w}_{0}\|_{\dot{B}^{-s}_{2,\infty}} (1+t)^{-\frac{s+\sigma_{c}-1}{2}}. \label{R-E566}
\end{eqnarray}
It follows from (\ref{R-E500}) and (\ref{R-E566}) that
\begin{eqnarray}
&&\sum_{q<0}\|\dot{\Delta}_{q}\Lambda^{\sigma_{c}-1}\tilde{w}\|_{L^2}
\nonumber\\&\leq&\sum_{q<0}
\|\dot{\Delta}_{q}\Lambda^{\sigma_{c}-1}[\mathcal{G}(t,x)\tilde{w}_{0}]\|_{L^2}
+\int^{t}_{0}\sum_{q<0}\|\dot{\Delta}_{q}\Lambda^{\sigma_{c}-1}[\mathcal{G}(t-\tau, x)\mathcal{R}(\tau)]\|_{L^2}d\tau.
\nonumber\\&\lesssim &  \|\tilde{w}_{0}\|_{\dot{B}^{-s}_{2,\infty}} (1+t)^{-\frac{s+\sigma_{c}-1}{2}}+
\int^{t}_{0}(1+t-\tau)^{-\frac{s+\sigma_{c}-1}{2}}\|\mathcal{R}(\tau)\|_{\dot{B}^{-s}_{2,\infty}}d\tau
\nonumber\\&\lesssim &  \|\tilde{w}_{0}\|_{\dot{B}^{-s}_{2,\infty}} (1+t)^{-\frac{s+\sigma_{c}-1}{2}}+
\int^{t}_{0}(1+t-\tau)^{-\frac{s+\sigma_{c}-1}{2}}\|\mathcal{R}(\tau)\|_{L^p}d\tau, \label{R-E577}
\end{eqnarray}
Just doing the same procedure leading to (\ref{R-E55}), we can obtain (\ref{R-E555}).

Hence, the proof of Lemma \ref{lem5.2} is complete.
\end{proof}

\begin{lem}\label{lem5.3}
 Under the assumption of Proposition \ref{prop5.2}, we have
\begin{eqnarray}
&&\sum_{q\geq0}2^{q(\sigma_{c}-1-\ell)}\|\Delta_{q}\Lambda^{\ell}\tilde{w}\|_{L^2}\nonumber\\&\lesssim& \|\tilde{w}_{0}\|_{B^{\sigma_{c}}_{2,1}}e^{-c_{1}t}+(1+t)^{-\frac{s+\ell}{2}}\mathcal{E}^{2}_{1}(t)
+(1+t)^{-\frac{s+\ell}{2}}\mathcal{E}_{0}(t)\mathcal{E}_{1}(t) \label{R-E68}
\end{eqnarray}
for $0\leq\ell\leq\sigma_{c}-1$.
\end{lem}

\begin{proof} Due to $\Delta_{q}f\equiv\dot{\Delta}_{q}f(q\geq0$), it is suffice to show (\ref{R-E68}) for the inhomogeneous case.
From (\ref{R-E36}), we have
\begin{eqnarray}
\|\Delta_{q}\Lambda^{\ell}\mathcal{G}(x,t)\tilde{w}_{0}\|_{L^2}\lesssim e^{-c_{2}t}\|\Delta_{q}\Lambda^{\ell}\tilde{w}_{0}\|_{L^2} \label{R-E69}
\end{eqnarray}
for all $q\geq0$.  Then it follows from (\ref{R-E50}) and (\ref{R-E69}) that
\begin{eqnarray}
&&\|\Delta_{q}\Lambda^{\ell}\tilde{w}\|_{L^2}
\nonumber\\&\leq&\|\Delta_{q}\Lambda^{\ell}[\mathcal{G}(x,t)\tilde{w}_{0}]\|_{L^2}+\int^{t}_{0}\|\Delta_{q}\Lambda^{\ell}[\mathcal{G}(x,t-\tau)\mathcal{R}(\tau)]\|_{L^2}d\tau
\nonumber\\&\lesssim& e^{-c_{1}t}\|\Delta_{q}\Lambda^{\ell}\tilde{w}_{0}\|_{L^2}+\int^{t}_{0}e^{-c_{1}(t-\tau)}\|\Delta_{q}\Lambda^{\ell}\mathcal{R}(\tau)\|_{L^2}d\tau
\nonumber\\&\lesssim& e^{-c_{1}t}\|\Delta_{q}\Lambda^{\ell}\tilde{w}_{0}\|_{L^2}+\|\Delta_{q}\Lambda^{\ell}\mathcal{R}(t)\|_{L^2} \label{R-E70}
\end{eqnarray}
which implies that
\begin{eqnarray}
\sum_{q\geq0}2^{q(\sigma_{c}-1-\ell)}\|\Delta_{q}\Lambda^{\ell}\tilde{w}\|_{L^2}\lesssim \|\tilde{w}_{0}\|_{B^{\sigma_{c}}_{2,1}}e^{-c_{1}t}+\|\Lambda^{\ell}\mathcal{R}(t)\|_{\dot{B}^{\sigma_{c}-1-\ell}_{2,1}} \label{R-E71}
\end{eqnarray}
for $0\leq\ell\leq\sigma_{c}-1$.

Using Propositions \ref{prop3.1}-\ref{prop3.2}, and Lemma \ref{lem3.2}, the nonlinear term in the right side of (\ref{R-E71}) can be estimated as follows:
\begin{eqnarray}
\|\Lambda^{\ell}\mathcal{R}_{1}(t)\|_{\dot{B}^{\sigma_{c}-1-\ell}_{2,1}}&\lesssim& \|\mathcal{H}_{1}(\tilde{w})\nabla\tilde{w}\|_{\dot{B}^{\sigma_{c}-1}_{2,1}}
\nonumber\\&\lesssim& \|\Lambda^{\ell}\tilde{w}\|_{\dot{B}^{\sigma_{c}-1-\ell}_{2,1}}\|\tilde{w}\|_{\dot{B}^{\sigma_{c}}_{2,1}}
\nonumber\\&\lesssim&
\begin{cases}
\|\Lambda^{\ell}\tilde{w}\|_{B^{\sigma_{c}-1-\ell}_{2,1}}\|\tilde{w}\|_{B^{\sigma_{c}}_{2,1}}, \ \ 0\leq\ell<\sigma_{c}-1;\\
\|\Lambda^{\sigma_{c}-1}\tilde{w}\|_{\dot{B}^{0}_{2,1}}\|\tilde{w}\|_{B^{\sigma_{c}}_{2,1}}, \ \ \ell=\sigma_{c}-1;
 \end{cases}
\nonumber\\&\lesssim& (1+t)^{-\frac{s+\ell}{2}}\mathcal{E}_{0}(t)\mathcal{E}_{1}(t).  \label{R-E72}
\end{eqnarray}
 Similarly, we have
\begin{eqnarray}
&&\|\Lambda^{\ell}\mathcal{R}_{2}(t)\|_{\dot{B}^{\sigma_{c}-1-\ell}_{2,1}}\nonumber\\&\leq & \|\Lambda^{\ell}\mathcal{R}_{21}(t)\|_{\dot{B}^{\sigma_{c}-1-\ell}_{2,1}}+
\|\Lambda^{\ell}\mathcal{R}_{22}(t)\|_{\dot{B}^{\sigma_{c}-1-\ell}_{2,1}}\nonumber\\&\lesssim&
\|\mathcal{H}_{2}(\tilde{w})\|_{\dot{B}^{\sigma_{c}-1}_{2,1}}\|(I-\mathcal{P})\tilde{w}\|_{\dot{B}^{\sigma_{c}-1}_{2,1}}
+\|\tilde{A}^{0}(\tilde{w})^{-1}\|_{\dot{B}^{\sigma_{c}-1}_{2,1}}\|\tilde{r}(\tilde{w})\|_{\dot{B}^{\sigma_{c}-1}_{2,1}}
\nonumber\\&\lesssim&
\begin{cases}
\|\Lambda^{\ell}\tilde{w}\|_{B^{\sigma_{c}-1-\ell}_{2,1}}\|\tilde{w}\|_{B^{\sigma_{c}}_{2,1}}+\|\Lambda^{\ell}\tilde{w}\|^2_{B^{\sigma_{c}-1-\ell}_{2,1}}, \ \ 0\leq\ell<\sigma_{c}-1;\\
\|\Lambda^{\sigma_{c}-1}\tilde{w}\|_{\dot{B}^{0}_{2,1}}\|\tilde{w}\|_{B^{\sigma_{c}}_{2,1}}+\|\Lambda^{\sigma_{c}-1}\tilde{w}\|^2_{\dot{B}^{0}_{2,1}}, \ \ \ell=\sigma_{c}-1;
 \end{cases}
\nonumber\\&\lesssim&(1+t)^{-\frac{s+\ell}{2}}\mathcal{E}_{0}(t)\mathcal{E}_{1}(t)+(1+t)^{-\frac{s+\ell}{2}}\mathcal{E}^{2}_{1}(t). \label{R-E73}
\end{eqnarray}
Then it follows from (\ref{R-E72})-(\ref{R-E73}) that
\begin{eqnarray}
\|\Lambda^{\ell}\mathcal{R}(t)\|_{\dot{B}^{\sigma_{c}-1-\ell}_{2,1}}\lesssim(1+t)^{-\frac{s+\ell}{2}}\mathcal{E}^{2}_{1}(t)
+(1+t)^{-\frac{s+\ell}{2}}\mathcal{E}_{0}(t)\mathcal{E}_{1}(t). \label{R-E74}
\end{eqnarray}
Finally, together with (\ref{R-E71}) and (\ref{R-E74}), we arrive at (\ref{R-E68}).
\end{proof}


\begin{lem}\label{lem5.4}
Under the assumption of Proposition \ref{prop5.2}, we have
\begin{eqnarray}
&&\|(I-\mathcal{P})\Lambda^{\ell}\tilde{w}\|_{B^{\sigma_{c}-2-\ell}_{2,1}}\nonumber\\&\lesssim& e^{-ct}\|(I-\mathcal{P})\tilde{w}_{0}\|_{B^{\sigma_{c}-2}_{2,1}}
+(1+\tau)^{-s-\frac{\ell+1}{2}}\mathcal{E}^2_{1}(t)
\nonumber\\&&+(1+\tau)^{-s-\frac{\ell+1}{2}}\mathcal{E}_{1}(t)\mathcal{E}_{2}(t), \label{R-E755}
\end{eqnarray}
for $0\leq\ell<\sigma_{c}-2$, and
\begin{eqnarray}
&&\|(I-\mathcal{P})\Lambda^{\sigma_{c}-2}\tilde{w}\|_{\dot{B}^{0}_{2,1}}\nonumber\\&\lesssim& e^{-ct}\|(I-\mathcal{P})\tilde{w}_{0}\|_{B^{\sigma_{c}-2}_{2,1}}
+(1+\tau)^{-s-\frac{\sigma_{c}-1}{2}}\mathcal{E}^2_{1}(t)
\nonumber\\&&+(1+\tau)^{-s-\frac{\sigma_{c}-1}{2}}\mathcal{E}_{1}(t)\mathcal{E}_{2}(t). \label{R-E75}
\end{eqnarray}
\end{lem}
\begin{proof}
Let $(V_{1},V_{2}) $ be the partition of $V$ according to the orthogonal decomposition $\mathbb{R}^{m}=\mathcal{M}\oplus\mathcal{M}^{\bot}$.
According to Theorem \ref{thm2.1}, we rewrite (\ref{R-E7}) as the normal form:
\begin{eqnarray}
\tilde{A}^{0}(V)V_{t}+\sum^{n}_{j=1}\tilde{A}^{j}(V)V_{x_{j}}+LV=\tilde{r}(V) \label{R-E76}
\end{eqnarray}
with
$$\tilde{A}^{0}(V)=\left(
          \begin{array}{cc}
            \tilde{A}^{0}_{11}(V) & 0 \\
            0 & \tilde{A}^{0}_{22}(V) \\
          \end{array}
        \right),\ \ \
    \tilde{A}^{j}(V)= \left(
          \begin{array}{cc}
            \tilde{A}^{j}_{11}(V) & \tilde{A}^{j}_{12}(V) \\
            \tilde{A}^{j}_{21}(V) & \tilde{A}^{j}_{22}(V) \\
          \end{array}
        \right) $$
$$L=\left(
          \begin{array}{cc}
            0 & 0 \\
            0 & L_{22} \\
          \end{array}
        \right),\ \ \
 \tilde{r}(V)=\left(
                \begin{array}{c}
                  0 \\
                   \tilde{r}_{12}(V) \\
                \end{array}
              \right).$$
Furthermore, we have
\begin{eqnarray}
\tilde{A}^{0}_{22}(V)V_{2t}+\sum^{n}_{j=1}[\tilde{A}^{j}_{21}(V)V_{1x_{j}}+\tilde{A}^{j}_{22}(V)V_{2x_{j}}]+L_{22}V_{2}=\tilde{r}_{12}(V). \label{R-E77}
\end{eqnarray}
 It is convenient to write (\ref{R-E77}) as
\begin{eqnarray}
\tilde{A}^{0}_{22}V_{2t}+L_{22}V_{2}=\mathcal{\tilde{R}}, \label{R-E78}
\end{eqnarray}
where $\tilde{A}^{0}_{22}=\tilde{A}^{0}_{22}(\bar{V})$ and
\begin{eqnarray*}\mathcal{\tilde{R}}&=&\tilde{A}^{0}_{22}\Big\{-\tilde{A}^{0}_{22}(V)^{-1}\sum^{n}_{j=1}[\tilde{A}^{j}_{21}(V)V_{1x_{j}}+\tilde{A}^{j}_{22}(V)V_{2x_{j}}]
\nonumber\\&&-[\tilde{A}^{0}_{22}(V)^{-1}-(\tilde{A}^{0}_{22})^{-1}]L_{22}V_{2}+\tilde{A}^{0}_{22}(V)^{-1}\tilde{r}_{12}(V)\Big\}
\nonumber\\&:=&\mathcal{\tilde{R}}_{1}+\mathcal{\tilde{R}}_{2}+\mathcal{\tilde{R}}_{3}.\end{eqnarray*}

Applying the operator $\Delta_{q}\Lambda^{\ell}(q\geq-1, \ 0\leq\ell<\sigma_{c}-2)$ to (\ref{R-E78}) gives
\begin{eqnarray}
\tilde{A}^{0}_{22}\Delta_{q}\Lambda^{\ell}V_{2t}+L_{22}\Delta_{q}\Lambda^{\ell}V_{2}=\Delta_{q}\Lambda^{\ell}\mathcal{\tilde{R}}. \label{R-E79}
\end{eqnarray}
Note that $(\tilde{A}^{0}_{22})^{-1}L_{22}$ is a positive definite matrix. Solving $\Delta_{q}\Lambda^{\ell}V_{2}$ from the ordinary differential equation (\ref{R-E79})
and taking the $L^2$-norm to get
\begin{eqnarray}
&&\|(I-\mathcal{P})\Delta_{q}\Lambda^{\ell}\tilde{w}\|_{L^2}\nonumber\\&\lesssim& e^{-ct}\|(I-\mathcal{P})\Delta_{q}\Lambda^{\ell}\tilde{w}_{0}\|_{L^2}
+\int^{t}_{0}e^{-c(t-\tau)}\|\Delta_{q}\Lambda^{\ell}\mathcal{\tilde{R}}(\tau)\|_{L^2}d\tau, \label{R-E80}
\end{eqnarray}
where $c$ is some positive constant. Then multiplying the factor $2^{q(\sigma_{c}-2-\ell)}$ on both sides of
(\ref{R-E80}) and summing up the resulting inequality, we immediately
deduce that
\begin{eqnarray}
&&\|(I-\mathcal{P})\Lambda^{\ell}\tilde{w}\|_{B^{\sigma_{c}-2-\ell}_{2,1}}\nonumber\\&\lesssim& e^{-ct}\|(I-\mathcal{P})\tilde{w}_{0}\|_{B^{\sigma_{c}-2}_{2,1}}
+\int^{t}_{0}e^{-c(t-\tau)}\|\mathcal{\tilde{R}}(\tau)\|_{B^{\sigma_{c}-2}_{2,1}}d\tau. \label{R-E81}
\end{eqnarray}
It follows from Lemma \ref{lem3.2} that
$\|\mathcal{\tilde{R}}(\tau)\|_{B^{\sigma_{c}-2}_{2,1}}=\|\mathcal{\tilde{R}}(\tau)\|_{\dot{B}^{\sigma_{c}-2}_{2,1}}+\|\mathcal{\tilde{R}}(\tau)\|_{L^2}$,
since $\sigma_{c}-2>0$. Using Proposition \ref{prop3.1} and Lemma \ref{lem3.1}, we have
\begin{eqnarray}
\|\mathcal{\tilde{R}}_{1}(\tau)\|_{\dot{B}^{\sigma_{c}-2}_{2,1}}&=&\|\mathcal{H}_{3}(V)\nabla V\|_{\dot{B}^{\sigma_{c}-2}_{2,1}}\nonumber\\&\lesssim&
\|\mathcal{H}_{3}(\tilde{w})\|_{\dot{B}^{\sigma_{c}-1}_{2,1}}\|\nabla \tilde{w}\|_{\dot{B}^{\sigma_{c}-2}_{2,1}}\nonumber\\&\lesssim& \|\tilde{w}\|_{B^{\sigma_{c}-1}_{2,1}}
\|\Lambda^{\ell+1} \tilde{w}\|_{B^{\sigma_{c}-1-(\ell+1)}_{2,1}}
\nonumber\\&\lesssim& (1+\tau)^{-s-\frac{\ell+1}{2}}\mathcal{E}^2_{1}(t), \label{R-E82}
\end{eqnarray}
where we used the constraint $\ell<\sigma_{c}-2$, and $\mathcal{H}_{3}(V)$ stands for $-\tilde{A}^{0}_{22}\tilde{A}^{0}_{22}(V)^{-1}\\ \tilde{A}^{j}_{21}(V)$ or $-\tilde{A}^{0}_{22}\tilde{A}^{0}_{22}(V)^{-1}\tilde{A}^{j}_{22}(V)$ in $\mathcal{\tilde{R}}_{1}$.

For $\mathcal{\tilde{R}}_{2}$, we have
\begin{eqnarray}
\|\mathcal{\tilde{R}}_{2}(\tau)\|_{\dot{B}^{\sigma_{c}-2}_{2,1}}&=&\|\mathcal{H}_{4}(\tilde{w})L_{22}V\|_{\dot{B}^{\sigma_{c}-2}_{2,1}}\nonumber\\&\lesssim&
\|\mathcal{H}_{4}(\tilde{w})\|_{\dot{B}^{\sigma_{c}-1}_{2,1}}\|(I-\mathcal{P})\tilde{w}\|_{\dot{B}^{\sigma_{c}-2}_{2,1}}\nonumber\\&\lesssim& \|\tilde{w}\|_{B^{\sigma_{c}-1}_{2,1}}
\|\Lambda^{\ell} (I-\mathcal{P})\tilde{w}\|_{B^{\sigma_{c}-2-\ell}_{2,1}}\nonumber\\&\lesssim& (1+\tau)^{-s-\frac{\ell+1}{2}}\mathcal{E}_{1}(t)\mathcal{E}_{2}(t), \label{R-E83}
\end{eqnarray}
where $\mathcal{H}_{4}(\tilde{w})=-\tilde{A}^{0}_{22}[\tilde{A}^{0}_{22}(V)^{-1}-(\tilde{A}^{0}_{22})^{-1}]$. Similarly, we can arrive at
\begin{eqnarray}
\|\mathcal{\tilde{R}}_{3}(\tau)\|_{\dot{B}^{\sigma_{c}-2}_{2,1}}\lesssim(1+\tau)^{-s-\frac{\ell+1}{2}}\mathcal{E}_{1}(t)\mathcal{E}_{2}(t). \label{R-E84}
\end{eqnarray}
Therefore, together with (\ref{R-E82})-(\ref{R-E84}), we can obtain
\begin{eqnarray}
\|\mathcal{\tilde{R}}(\tau)\|_{\dot{B}^{\sigma_{c}-2}_{2,1}}\lesssim(1+\tau)^{-s-\frac{\ell+1}{2}}\mathcal{E}^2_{1}(t)
+(1+\tau)^{-s-\frac{\ell+1}{2}}\mathcal{E}_{1}(t)\mathcal{E}_{2}(t). \label{R-E85}
\end{eqnarray}
In addition, from (\ref{R-E9}), it is not difficult to get
\begin{eqnarray}
\|\mathcal{\tilde{R}}(\tau)\|_{L^2}&\lesssim &\|\mathcal{H}_{3}(w)\|_{L^\infty}\|\nabla \tilde{w}\|_{L^2}+\|\tilde{w}\|_{L^\infty}\|(I-\mathcal{P})\tilde{w}\|_{L^2}
\nonumber\\&\lesssim& \|\Lambda^{\ell}\tilde{w}\|_{\dot{B}^{\sigma_{c}-1-\ell}_{2,1}}\Big(\|\nabla \tilde{w}\|_{B^{\sigma_{c}-2}_{2,1}}+\|(I-\mathcal{P})\tilde{w}\|_{B^{\sigma_{c}-2}_{2,1}}\Big)
\nonumber\\&\lesssim& (1+\tau)^{-s-\frac{\ell+1}{2}}\mathcal{E}^2_{1}(t)
+(1+\tau)^{-s-\frac{\ell+1}{2}}\mathcal{E}_{1}(t)\mathcal{E}_{2}(t). \label{R-E86}
\end{eqnarray}
Combining (\ref{R-E81}) and (\ref{R-E85})-(\ref{R-E86}), we conclude that (\ref{R-E755}).

In case that $\ell=\sigma_{c}-2$, by applying the operator $\dot{\Delta}_{q}\Lambda^{\sigma_{c}-2}(q\in \mathbb{Z})$ to (\ref{R-E78}) and performing the similar procedure leading to
(\ref{R-E81}), we obtain
\begin{eqnarray}
&&\|(I-\mathcal{P})\Lambda^{\sigma_{c}-2}\tilde{w}\|_{\dot{B}^{0}_{2,1}}\nonumber\\&\lesssim& e^{-ct}\|(I-\mathcal{P})\tilde{w}_{0}\|_{B^{\sigma_{c}-2}_{2,1}}
+\int^{t}_{0}e^{-c(t-\tau)}\|\mathcal{\tilde{R}}(\tau)\|_{\dot{B}^{\sigma_{c}-2}_{2,1}}d\tau. \label{R-E811}
\end{eqnarray}
Next, we revise the inequalities (\ref{R-E82})-(\ref{R-E84}) as follows:
\begin{eqnarray}
\|\mathcal{\tilde{R}}_{1}(\tau)\|_{\dot{B}^{\sigma_{c}-2}_{2,1}}&\lesssim&
\|\mathcal{H}_{3}(\tilde{w})\|_{\dot{B}^{\sigma_{c}-1}_{2,1}}\|\nabla \tilde{w}\|_{\dot{B}^{\sigma_{c}-2}_{2,1}}\nonumber\\&\lesssim& \|\tilde{w}\|_{B^{\sigma_{c}-1}_{2,1}}
\|\Lambda^{\sigma_{c}-1} \tilde{w}\|_{\dot{B}^{0}_{2,1}}
\nonumber\\&\lesssim& (1+\tau)^{-s-\frac{\sigma_{c}-1}{2}}\mathcal{E}^2_{1}(t), \label{R-E812}
\end{eqnarray}
\begin{eqnarray}
\|\mathcal{\tilde{R}}_{2}(\tau)\|_{\dot{B}^{\sigma_{c}-2}_{2,1}}&\lesssim&
\|\mathcal{H}_{4}(\tilde{w})\|_{\dot{B}^{\sigma_{c}-1}_{2,1}}\|(I-\mathcal{P})\tilde{w}\|_{\dot{B}^{\sigma_{c}-2}_{2,1}}\nonumber\\&\lesssim& \|\tilde{w}\|_{B^{\sigma_{c}-1}_{2,1}}
\|\Lambda^{\sigma_{c}-2} (I-\mathcal{P})\tilde{w}\|_{\dot{B}^{0}_{2,1}}\nonumber\\&\lesssim& (1+\tau)^{-s-\frac{\sigma_{c}-1}{2}}\mathcal{E}_{1}(t)\mathcal{E}_{2}(t), \label{R-E813}
\end{eqnarray}
and
\begin{eqnarray}
\|\mathcal{\tilde{R}}_{3}(\tau)\|_{\dot{B}^{\sigma_{c}-2}_{2,1}}\lesssim(1+\tau)^{-s-\frac{\sigma_{c}-1}{2}}\mathcal{E}_{1}(t)\mathcal{E}_{2}(t), \label{R-E844}
\end{eqnarray}
which lead to (\ref{R-E75}) directly.
\end{proof}

\noindent \textbf{\textit{The proofs of Propositions \ref{prop5.1}-\ref{prop5.2}}.}
From Lemmas \ref{lem5.2}-\ref{lem5.3}, we get
\begin{eqnarray}
\mathcal{E}_{1}(t)\lesssim E_{0}+\mathcal{E}^{2}(t)+\mathcal{E}_{0}(t)\mathcal{E}_{1}(t). \label{R-E888}
\end{eqnarray}
From Lemma \ref{lem5.4}, we get
\begin{eqnarray}
\mathcal{E}_{2}(t)\lesssim E_{0}+\mathcal{E}^{2}_{1}(t)+\mathcal{E}_{1}(t)\mathcal{E}_{2}(t). \label{R-E889}
\end{eqnarray}
The time-weighted energy inequalities (\ref{R-E888})-(\ref{R-E889}) implies (\ref{R-E54}) immediately, so
the proof of Proposition \ref{prop5.2} is finished ultimately.

Furthermore, from Theorem \ref{thm2.3}, we see that
$\mathcal{E}_{0}(t)\lesssim \|V_{0}-\bar{V}\|_{B^{\sigma_{c}}_{2,1}}\lesssim E_{0}$. Thus, if $E_{0}$ is sufficient small, it follows from (\ref{R-E54}) that
\begin{eqnarray}
\mathcal{E}(t)\lesssim E_{0}+\mathcal{E}^{2}(t), \label{R-E88}
\end{eqnarray}
which can deduce that $\mathcal{E}(t)\lesssim E_{0}$, provided that $E_{0}$ is sufficient small. Consequently, we obtain the decay
estimates in Proposition \ref{prop5.1}. $\square$

Based on Propositions \ref{prop4.1}-\ref{prop4.2} and \ref{prop5.1}, we have a analogue decay estimates on the framework of $B^{\sigma_{c}}_{2,1}\cap L^{p}(1\leq p<2)$.
\begin{prop}\label{prop5.3}
Let $\tilde{w}=V(t,x)-\bar{V}$ be the global classical solution in the sense of Theorem \ref{thm2.3}. Suppose that $\tilde{w}_{0}-\bar{w}\in B^{\sigma_{c}}_{2,1}\cap L^{p}(1\leq p<2)$ and the norm
$\widetilde{E}_{0}:=\|w_{0}-\bar{w}\|_{B^{\sigma_{c}}_{2,1}\cap L^{p}}$ is sufficiently small. Then it holds that
\begin{eqnarray}
\|\Lambda^{\ell}\tilde{w}(\tau)\|_{X_{1}}\lesssim \widetilde{E}_{0}(1+t)^{-\frac{n}{2}(\frac{1}{p}-\frac{1}{2})-\frac{\ell}{2}}\label{R-E89}
\end{eqnarray}
for $0\leq\ell\leq \sigma_{c}-1$, and
\begin{eqnarray}
\|\Lambda^{\ell}(I-\mathcal{P})\tilde{w}(\tau)\|_{X_{2}}\lesssim \widetilde{E}_{0}(1+t)^{-\frac{n}{2}(\frac{1}{p}-\frac{1}{2})-\frac{\ell+1}{2}} \label{R-E90}
\end{eqnarray}
for $0\leq\ell\leq \sigma_{c}-2$, where $X_{1}$ and $X_{2}$ are the same space notations as in Proposition \ref{prop5.1}.
\end{prop}

As a direct consequence of Propositions \ref{prop5.1}-\ref{prop5.2}, the optimal decay estimates in the usual $L^2$ space are available.
\begin{cor}\label{cor5.1}
Let $\tilde{w}=V(t,x)-\bar{V}$ be the global classical solution in the sense of Theorem \ref{thm2.3}.  \begin{itemize}
\item [(i)]  If $E_{0}$ is sufficiently small, then
\begin{eqnarray}
\|\Lambda^{\ell}\tilde{w}\|_{L^2}\lesssim E_{0}(1+t)^{-\frac{\ell+s}{2}},  \  0\leq\ell\leq\sigma_{c}-1; \label{R-E91}
\end{eqnarray}
\begin{eqnarray}
\|\Lambda^{\ell}(I-\mathcal{P})\tilde{w}\|_{L^2}\lesssim E_{0}(1+t)^{-\frac{s+\ell+1}{2}}, \  0\leq\ell\leq \sigma_{c}-2. \label{R-E92}
\end{eqnarray}

\item [(ii)] If $\widetilde{E}_{0}$ is sufficiently small, then
\begin{eqnarray}
\|\Lambda^{\ell}\tilde{w}\|_{L^2}\lesssim \widetilde{E}_{0}(1+t)^{-\frac{n}{2}(\frac{1}{p}-\frac{1}{2})-\frac{\ell}{2}},  0\leq\ell\leq\sigma_{c}-1; \label{R-E93}
\end{eqnarray}
\begin{eqnarray}
\|\Lambda^{\ell}(I-\mathcal{P})\tilde{w}\|_{L^2}\lesssim \widetilde{E}_{0}(1+t)^{-\frac{n}{2}(\frac{1}{p}-\frac{1}{2})-\frac{\ell+1}{2}},\  0\leq\ell\leq \sigma_{c}-2. \label{R-E94}
\end{eqnarray}
\end{itemize}
\end{cor}

Moreover, the optimal $L^{p}$-$L^{q}$ decay rates for solutions are shown by Corollary \ref{cor5.1} and Lemma \ref{lem8.4}.

\begin{cor}\label{cor5.2}
Suppose that $\tilde{w}_{0}\in B^{\sigma}_{2,1}\cap L^{p}(1\leq p<2)$ and $\widetilde{E}_{0}$ is sufficiently small. Then the solution and various derivatives decay in the $L^{q}$ norm:
\begin{eqnarray}
\|\Lambda^{k}\tilde{w}\|_{L^{q}}\lesssim \widetilde{E}_{0} (1+t)^{-\gamma_{p,q}-\frac{k}{2}} \label{R-E95}
\end{eqnarray}
 for $2\leq q\leq\infty$ and $0\leq k\leq\sigma_{c}-1-2\gamma_{2,q}$,
and
\begin{eqnarray}
\|\Lambda^{k}(I-\mathcal{P})\tilde{w}\|_{L^{q}}\lesssim \widetilde{E}_{0} (1+t)^{-\gamma_{p,q}-\frac{k+1}{2}}, \label{R-E96}
\end{eqnarray}
for $2\leq q\leq n$ and $0\leq k\leq\sigma_{c}-2-2\gamma_{2,q}$, where $\gamma_{p,q}=\frac{n}{2}(\frac{1}{p}-\frac{1}{q})$ is the $L^p$-$L^q$ decay rate of heat kernel.
\end{cor}
\begin{proof}
From Lemma \ref{lem8.4} (taking $r=2$), without loss of generality suppose that $m<\varrho$, then we have
\begin{eqnarray}
\|\Lambda^{k}\tilde{w}\|_{L^{q}(R^{n})} &\lesssim &\|\Lambda^{m}f\|_{L^2}^{1-\theta}\|\Lambda^{\varrho}w\|^{\theta}_{L^2(R^{n})}
\nonumber\\& \lesssim & \Big\{(1+t)^{-\frac{n}{2}(\frac{1}{p}-\frac{1}{2})-\frac{m}{2}}\Big\}^{1-\theta}\Big\{(1+t)^{-\frac{n}{2}(\frac{1}{p}-\frac{1}{2})-\frac{\varrho}{2}}\Big\}^{\theta}
\nonumber\\&=& (1+t)^{-\frac{n}{2}(\frac{1}{p}-\frac{1}{2})-\frac{m}{2}(1-\theta)-\frac{\varrho}{2}\theta}
\nonumber\\&=& (1+t)^{-\frac{n}{2}(\frac{1}{p}-\frac{1}{2})-\frac{k}{2}-\frac{n}{2}(\frac{1}{2}-\frac{1}{q})}
\nonumber\\&=& (1+t)^{-\frac{n}{2}(\frac{1}{p}-\frac{1}{q})-\frac{k}{2}} \label{R-E97}
\end{eqnarray}
for $k\geq0$, where we have used the relation $k+n(\frac{1}{2}-\frac{1}{q})=m(1-\theta)+\varrho\theta$.
Finally, it follows from $0\leq\theta\leq1$  and $0\leq m<\varrho\leq \sigma_{c}-1$ that $$0\leq k\leq\sigma_{c}-1-n\Big(\frac{1}{2}-\frac{1}{q}\Big).$$
Similarly, the corresponding decay rates for $(I-\mathcal{P})\tilde{w}$ are also available.
\end{proof}

\section{Applications}\setcounter{equation}{0}\label{sec:6}
Our decay results for generally dissipative systems have a great
potential for applications, since there are many concrete models satisfy the entropy assumption and [SK] condition.
In this section, we present an application to the gas dynamics with relaxation, say, the damped compressible Euler equations, which
are given by
\begin{equation}
\left\{
\begin{array}{l}
\partial_{t}\rho + \nabla\cdot(\rho\textbf{u}) = 0 , \\
\partial_{t}(\rho\textbf{u}) +\nabla\cdot(\rho\textbf{u}\otimes\textbf{u}) +
\nabla p(\rho) =-\rho\textbf{u}.
\end{array} \right.\label{R-E98}
\end{equation}
Here $\rho = \rho(t, x)$ is the fluid density function of
$(t,x)\in[0,+\infty)\times\mathbb{R}^{3}$;
$\textbf{u}=\textbf{u}(t, x)=(u^1,u^2,u^{3})^{\top}$ denotes the
fluid velocity. The pressure $P$ is related to the density by
$p(\rho)$, which satisfies the classical assumption
$$p'(\rho)>0,\ \ \ \forall\rho>0.$$
An usual simplicity $p(\rho):=\rho^{\gamma}(\gamma\geq 1)$, where
the adiabatic exponent $\gamma>1$ corresponds to the isentropic flow
and $\gamma=1$ corresponds to the isothermal flow, see for example \cite{STW}.

In the present paper, we investigate the Cauchy problem of 3D compressible
Euler equations (\ref{R-E98}) with the initial condition:
\begin{equation}
(\rho,\textbf{u})(0,x)=(\rho_{0},\textbf{u}_{0}).\label{R-E99}
\end{equation}

We are interested in the damping effect on the regularity and large-time
behavior of classical solutions of (\ref{R-E98})-(\ref{R-E99}).
For the one-dimensional Euler equations with damping, the global existence of a smooth
solution with small data was obtained first by Nishida \cite{N2}. The large-time
behavior was shown by many papers, see, e.g., the excellent survey paper by Dafermos
\cite{D1}, the book by Hsiao \cite{H} and references therein. Here, we consider
the multi-dimensional case. It has been shown by \cite{STW} and \cite{WY} that
the damping term could prevent the development singularities and the Cauchy problem
(\ref{R-E98})-(\ref{R-E99}) has a unique classical solution which decays in the $L^2$-norm
to the constant background state at the rate of $(1+t)^{-3/4}$ in \cite{STW}, and in the
$L^{p}(1<p\leq\infty)$-norm at the rate of $(1+t)^{-n/2(1-1/p)}$ in \cite{WY}, respectively.

Recently, Tan and Wu \cite{TW} performed the spectral analysis to improve the above decay rates such that
the density converges to its equilibrium state at the rates $(1+t)^{-\frac{3}{4}-\frac{s}{2}}$ in the $L^2$-norm,
and the momentum of the system decays at the rates $(1+t)^{-\frac{5}{4}-\frac{s}{2}}$ in the $L^2$-norm, as the initial data
$(\rho_{0},\textbf{u}_{0})\in H^{l}\cap \dot{B}^{-s}_{1,\infty}(l\geq4,\ s\in [0,1])$. At this stage, we give the optimal
decay rates on the framework of spatially critical Besov spaces $B^{\sigma_{c}}_{2,1}\cap \dot{B}^{-s}_{2,\infty}(\sigma_{c}=5/2,\ s\in (0,3/2])$.

First, let us mention that (\ref{R-E98}) is a class of dissipative hyperbolic equations which satisfies the entropy assumption and [SK] condition,
see \cite{XK} for the rigorous verification. Therefore, we have the following global-in-time existence result in the critical space with $\sigma_{c}=5/2$.

\begin{thm}\cite{XK} \label{thm6.1}
Let $\bar{\rho}>0$ be a constant reference density. Suppose that \
$\rho_{0}-\bar{\rho}$ and $\mathbf{m}_{0}\in B^{\sigma_{c}}_{2,1}$,
there exists a positive constant $\tilde{\delta}_{0}$ such that if
$$\|(\rho_{0}-\bar{\rho},\mathbf{m}_{0})\|_{B^{\sigma_{c}}_{2,1}}\leq
\tilde{\delta}_{0}$$ with $\mathbf{m}_{0}=\rho_{0}\mathbf{u}_{0}$,
then the Cauchy problem (\ref{R-E98})-(\ref{R-E99}) has a unique
global solution $(\rho,\mathbf{m})\in \mathcal{C}^{1}(\mathbb{R}^{+}\times \mathbb{R}^{3})$ satisfying
\begin{eqnarray*}
(\rho-\bar{\rho},\mathbf{m}) \in
\widetilde{\mathcal{C}}(B^{\sigma_{c}}_{2,1})\cap
\widetilde{\mathcal{C}}^1(B^{\sigma_{c}-1}_{2,1}).
\end{eqnarray*}
Moreover, there exist two positive constants $\tilde{C}_{0}$ and $\tilde{\mu}_{0}$ such that the following energy inequality holds
\begin{eqnarray}
&&\|(\rho-\bar{\rho},\mathbf{m})\|_{\widetilde{L}^\infty(B^{\sigma_{c}}_{2,1})}
+\tilde{\mu}_{0}\Big(\|\mathbf{m}\|_{\widetilde{L}^2(B^{\sigma_{c}}_{2,1})}
+\|(\nabla\rho,\nabla\mathbf{m})\|_{\widetilde{L}^2(B^{\sigma_{c}-1}_{2,1})}\Big)
\nonumber\\&\leq& \tilde{C}_{0}\|(\rho_{0}-\bar{\rho},
\mathbf{m}_{0})\|_{B^{\sigma_{c}}_{2,1}}. \label{R-E100}
\end{eqnarray}
\end{thm}

Based on Theorem \ref{thm6.1}, those decay results for general dissipative systems can be applied to (\ref{R-E98})-(\ref{R-E99}), and $X_{1}$ and $X_{2}$ are the same space notations with $\sigma_{c}=5/2$.
Precisely,
\begin{thm}\label{thm6.2}
Let $(\rho, \mathbf{m})(t,x)$ be the global classical solutions of Theorem \ref{thm6.1}. If further the initial data $(\rho_{0}-\bar{\rho},
\mathbf{m}_{0})\in \dot{B}^{-s}_{2,\infty}(0<s\leq 3/2)$ and
$$E_{0}:=\|(\rho_{0}-\bar{\rho}, \mathbf{m}_{0})\|_{B^{\sigma_{c}}_{2,1}\cap\dot{B}^{-s}_{2,\infty}}$$
is sufficiently small. Then the classical solutions $(\rho, \mathbf{m})(t,x)$  satisfies the decay estimates
\begin{eqnarray}
\|\Lambda^{\ell}(\rho-\bar{\rho},\mathbf{m})\|_{X_{1}}\lesssim E_{0}(1+t)^{-\frac{s+\ell}{2}} \label{R-E101}
\end{eqnarray}
for $0\leq \ell\leq \sigma_{c}-1$
and
\begin{eqnarray}
\|\Lambda^{\ell}\mathbf{m}\|_{X_{2}}\lesssim E_{0}(1+t)^{-\frac{s+\ell+1}{2}} \label{R-E102}
\end{eqnarray}
for $0\leq \ell\leq \sigma_{c}-2$.
\end{thm}

Of course, we have also a analogue decay estimates
on the framework of $B^{\sigma_{c}}_{2,1}\cap L^p$.

\begin{thm}\label{thm6.3}
Let $(\rho, \mathbf{m})(t,x)$ be the global classical solutions of Theorem \ref{thm6.1}. If further the initial data $(\rho_{0}-\bar{\rho},
\mathbf{m}_{0})\in L^p(1\leq p<2)$ and
$$\widetilde{E}_{0}:=\|(\rho_{0}-\bar{\rho},
\mathbf{m}_{0})\|_{B^{\sigma_{c}}_{2,1}\cap L^p}$$
is sufficiently small. Then the classical solutions $(\rho,\mathbf{m})$  satisfies the following optimal decay estimates
\begin{eqnarray}
\|\Lambda^{\ell}(\rho-\bar{\rho},\mathbf{m})\|_{X_{1}}\lesssim \tilde{E}_{0}(1+t)^{-\frac{3}{2}(\frac{1}{p}-\frac{1}{2})-\frac{\ell}{2}} \label{R-E103}
\end{eqnarray}
for $0\leq \ell\leq \sigma_{c}-1$
and
\begin{eqnarray}
\|\Lambda^{\ell}\mathbf{m}\|_{X_{2}}\lesssim \tilde{E}_{0}(1+t)^{-\frac{n}{2}(\frac{1}{p}-\frac{1}{2})-\frac{\ell+1}{2}} \label{R-E104}
\end{eqnarray}
for $0\leq \ell\leq \sigma_{c}-2$.
\end{thm}

As a direct consequence of Theorems \ref{thm6.2}-\ref{thm6.3}, the optimal decay estimates in the $L^2$ are available.
\begin{cor} \label{cor6.1}
Let $(\rho,\mathbf{m})$ be the global classical solutions of Theorem \ref{thm6.1}.
\begin{itemize}
\item [(i)]  If $E_{0}$ is sufficiently small, then
\begin{eqnarray}
\|\Lambda^{\ell}(\rho-\bar{\rho},\mathbf{m})\|_{L^2}\lesssim E_{0}(1+t)^{-\frac{\ell+s}{2}},  \ \ \ 0\leq \ell\leq\sigma_{c}-1; \label{R-E105}
\end{eqnarray}
\begin{eqnarray}
\|\Lambda^{\ell}\mathbf{m}\|_{L^2}\lesssim  E_{0}(1+t)^{-\frac{s+\ell+1}{2}}, \ \ \ 0\leq \ell\leq \sigma_{c}-2. \label{R-E106}
\end{eqnarray}

\item [(ii)] If $\widetilde{E}_{0}$ is sufficiently small, then
\begin{eqnarray}
\|\Lambda^{\ell}(\rho-\bar{\rho},\mathbf{m})\|_{L^2}\lesssim \widetilde{E}_{0}(1+t)^{-\frac{3}{2}(\frac{1}{p}-\frac{1}{2})-\frac{\ell}{2}}, \ \ \  0\leq \ell\leq\sigma_{c}-1; \label{R-E107}
\end{eqnarray}
\begin{eqnarray}
\|\Lambda^{\ell}\mathbf{m}\|_{L^2}\lesssim \widetilde{E}_{0}(1+t)^{-\frac{3}{2}(\frac{1}{p}-\frac{1}{2})-\frac{\ell+1}{2}},\ \ \ 0\leq \ell\leq \sigma_{c}-2. \label{R-E108}
\end{eqnarray}
\end{itemize}
\end{cor}

Finally, we have the optimal $L^{p}$-$L^{q}$ decay estimates for (\ref{R-E98})-(\ref{R-E99}).
\begin{cor}\label{cor6.2}
Suppose that $U_{0}-\bar{U}\in B^{\sigma_{c}}_{2,1}\cap L^{p}(1\leq p<2)$ and $ \widetilde{E}_{0}$ is sufficiently small. Then the solution and various derivatives decay in the $L^{q}$ norm:
\begin{eqnarray}
\|\Lambda^{k}(\rho-\bar{\rho},\mathbf{m})\|_{L^{q}}\lesssim \widetilde{E}_{0} (1+t)^{-\gamma_{p,q}-\frac{k}{2}} \label{R-E109}
\end{eqnarray}
for $2\leq q\leq\infty$ and $0\leq k\leq\sigma_{c}-1-2\gamma_{2,q}$, and
\begin{eqnarray}
\|\Lambda^{k}\mathbf{m}\|_{L^{q}}\lesssim \widetilde{E}_{0} (1+t)^{-\gamma_{p,q}-\frac{k+1}{2}} \label{R-E110}
\end{eqnarray}
for $2\leq q\leq 3$ and $0\leq k\leq\sigma_{c}-2-2\gamma_{2,q}$, where $\gamma_{p,q}:=\frac{3}{2}(\frac{1}{p}-\frac{1}{q})$.
\end{cor}

\begin{rem}
It is worth noting that we first obtain the various decay rates for $(\rho, \mathbf{m})(t,x)$ and its derivatives not only on the framework of spatially critical Besov spaces, but also the derivative index can take values in some interval, which improve those optimal decay results in \cite{STW,TW,WY} heavily. For instance, from Corollaries \ref{cor6.1}-\ref{cor6.2}, we can deduce that the well-known decay rates for 3D damped compressible Euler equations:
$$\|\rho-\bar{\rho}\|_{L^2}\lesssim (1+t)^{-\frac{3}{4}}, \ \ \ \|\nabla\rho\|_{L^2}\lesssim (1+t)^{-\frac{5}{4}},$$
$$\|(\rho-\bar{\rho},\mathbf{m})\|_{L^\infty}\lesssim (1+t)^{-\frac{3}{2}}, \ \ \ \|\mathbf{m}\|_{L^2}\lesssim (1+t)^{-\frac{5}{4}}.$$
Obviously, we see that the momentum has an extra time-decay $(1+t)^{-1/2}$ in $L^2$, which is just the recent decay result of \cite{TW}.
\end{rem}

\section{Appendix A (decay framework)} \setcounter{equation}{0}\label{sec:7}
This section can be regarded as an independent one in regard to the present paper, which is the main motivation of our work,
however, we would like to supplement it for completeness. Additionally, allow us to abuse the notation for constants a little.

\subsection{The damped symmetric hyperbolic system}
First, we give the decay framework for the damped symmetric hyperbolic system, which reads as
\begin{equation}
\left\{
\begin{array}{l}
A^{0}\partial_{t}w + \sum_{j=1}^{n}A^{j}w_{x_{j}}+Lw = 0 , \\
w(0,x)=w_{0},
\end{array} \right.\label{R-E111}
\end{equation}
where $t\geq0, x=(x_{1},x_{2},\cdot\cdot\cdot,x_{n})\in \mathbb{R}^{n}$ and $w=w(t,x)$ is an $\mathbb{R}^{N}$-valued function; $L, A^{j}(j=0,1,2,\cdot\cdot\cdot,n)$ are constant matrices of order $N$.

We assume that the equation of (\ref{R-E111}) is ``symmetric hyperbolic" in the same sense as in Definition \ref{defn2.2}:
\begin{itemize}
\item[(A1)] Matrices $A^{j}(j=0,\cdot\cdot\cdot,n)$ are real symmetric and, in addition, $A^{0}$ is positive
definite. $L$ is real symmetric and nonnegative definite, and its null space coincides with $\mathcal{M}$.
\end{itemize}
Furthermore, we also assume (\ref{R-E111}) satisfies the [SK] condition as in Definition \ref{defn2.4}:
\begin{itemize}
\item[(A2)] Let $\phi\in
\mathbb{R}^{N}$ satisfies $\phi\in\mathcal{M}$ (i.e., $L\phi=0$) and
$\lambda A^{0}+A(\omega)\phi=0$ for some
$(\lambda,\omega)\in \mathbb{R}\times\mathbb{S}^{n-1}$, then
$\phi=0$.
\end{itemize}
Then the result on the decay estimates of solution to the system (\ref{R-E111}) is stated as follows.
\begin{prop}\label{prop7.1}
Let the assumptions (A1)-(A2) hold.
Suppose $w_{0}\in L^2(\mathbb{R}^{n})\\ \cap \dot{B}^{-s}_{2,\infty}(\mathbb{R}^{n})$ for $s>0$, then the solution of (\ref{R-E111}) has the decay estimate
\begin{equation}
\|w\|_{L^2(\mathbb{R}^{n})}\lesssim (1+t)^{-s/2}. \label{R-E112}
\end{equation}
In particular, suppose $w_{0}\in L^2(\mathbb{R}^{n})\cap L^p(\mathbb{R}^{n})(1\leq p<2$), one further has
\begin{equation}
\|w\|_{L^2(\mathbb{R}^{n})}\lesssim (1+t)^{-\frac{n}{2}(\frac{1}{p}-\frac{1}{2})}. \label{R-E113}
\end{equation}
\end{prop}
\begin{proof}
\noindent\textit{\underline{Step 1.}}
Perform the Fourier transform of (\ref{R-E111}) to give
\begin{equation}
A^{0}\hat{w}_{t}+\big(i|\xi|A(\omega)+L\Big)\hat{w}=0, \label{R-E114}
\end{equation}
where $A(\omega):=\sum_{j=1}^{n}A^{j}\omega_{j}$.
By performing the inner product of (\ref{R-E114}) with $\hat{w}$ and taking
the real part of the resulting equation, we get
\begin{equation}
\frac{1}{2}\frac{d}{dt}(A^{0}\hat{w},\hat{w})_{t}+c_{0}|(I-\mathcal{P})\hat{w}|^2\leq 0, \label{R-E115}
\end{equation}
where we have noticed that $A^{0}, A(\omega)$ and $L$ are real symmetric. $\mathcal{P}$ is the orthogonal projection onto $\mathcal{M}=\mathrm{Ker}L$ and $c_{0}>0$ is some constant.
\\

\noindent\textit{\underline{Step 2.}}
Multiplying (\ref{R-E114}) by $-i|\xi|K(\omega)$, performing the inner product with $\hat{w}$ and then taking the real part of the resulting equality, we arrive at
\begin{eqnarray}
&&\frac{1}{2}\frac{d}{dt}\mathrm{Im}(|\xi|K(\omega)A^{0}\hat{w},\hat{w})+|\xi|^2([K(\omega)A(\omega)]'\hat{w},\hat{w})\nonumber\\&=&|\xi|\mathrm{Im}(K(\omega)L\hat{w},\hat{w}). \label{R-E116}
\end{eqnarray}
It follows from Theorem \ref{thm2.2} that $[K(\omega)A(\omega)]'+L$ is positive definite, so there exists a constant $c_{1}>0$ such that
\begin{eqnarray}
|\xi|^2([K(\omega)A(\omega)]'\hat{w},\hat{w})\geq c_{1}|\xi|^2|\hat{w}|^2-|\xi|^2|(I-\mathcal{P})\hat{w}|^2. \label{R-E117}
\end{eqnarray}
From Young's inequality, the right side of (\ref{R-E116}) can be estimated as
\begin{eqnarray}
\Big||\xi|\mathrm{Im}(K(\omega)L\hat{w},\hat{w})\Big|\leq \epsilon |\xi|^2|\hat{w}|^2+C({\epsilon})|(I-\mathcal{P})\hat{w}|^2. \label{R-E118}
\end{eqnarray}
 Together with (\ref{R-E116})-(\ref{R-E118}), we are led to the estimate
\begin{eqnarray}
\frac{1}{2}\frac{d}{dt}\mathrm{Im}\Big(\frac{|\xi|}{1+|\xi|^2}K(\omega)A^{0}\hat{w},\hat{w}\Big)+\frac{c_{1}}{2}\frac{|\xi|^2}{1+|\xi|^2}|\hat{w}|^2\leq C|(I-\mathcal{P})\hat{w}|^2, \label{R-E119}
\end{eqnarray}
where we chosen $\epsilon>0$ satisfying $\epsilon\leq c_{1}/2$.

Next, we multiply (\ref{R-E119}) by the constant $\kappa>0$ and add the resulting inequality and (\ref{R-E115}) to get
\begin{eqnarray}
\frac{d}{dt}E[\hat{w}]+(c_{0}-\kappa C)|(I-\mathcal{P})\hat{w}|^2+\frac{c_{1}\kappa|\xi|^2}{1+|\xi|^2}|\hat{w}|^2\leq0\label{R-E120}
\end{eqnarray}
with
\begin{eqnarray*}
E[\hat{w}]=\frac{1}{2}(A^{0}\hat{w},\hat{w})+\frac{\kappa}{2}\mathrm{Im}\Big(\frac{|\xi|}{1+|\xi|^2}K(\omega)A^{0}\hat{w},\hat{w}\Big),
\end{eqnarray*}
where we chosen $\kappa>0$ so small that $c_{1}-\kappa C\geq0$ and $E[\hat{w}]\approx|\hat{w}|^2$, since $A^{0}$ is positive definite.

Let us consider the unit decomposition: $1\equiv\phi(\xi)+\varphi(\xi)$, where $\phi,\varphi\in C^{\infty}_{c}(\mathbb{R}^{n})$ ($0\leq\phi(\xi), \varphi(\xi)\leq1$) satisfy
$$
\phi(\xi)\equiv1,\   \mbox{if}\   |\xi|\leq R;  \ \ \phi(\xi)\equiv0,  \ \mbox{if} \ |\xi|\geq 2R
$$
with $R>0$. Furthermore, set $w_{1}=\mathcal{F}^{-1}[\phi(\xi)\hat{w}(\xi)]$ and $w_{2}=\mathcal{F}^{-1}[\varphi(\xi)\hat{w}(\xi)]$, so we have $w=w_{1}+w_{2}.$

\underline{\textit{Case 1}(high-frequency)}

Multiplying the inequality (\ref{R-E120}) by $\varphi^2$, we have
\begin{eqnarray}
\frac{d}{dt}(\varphi^2E[\hat{w}])+\frac{c_{1}R^2}{1+R^2}|\varphi\hat{w}|^2\leq0,\label{R-E121}
\end{eqnarray}
which implies
\begin{eqnarray}
\|w_{2}\|_{L^2}\leq Ce^{-c_{2}t}\|w_{0}\|_{L^2},\label{R-E122}
\end{eqnarray}
where the constant $c_{2}>0$ depends on $R$.

\underline{\textit{Case 2}(low-frequency)}

Multiplying the inequality (\ref{R-E120}) by $\phi^2$, we have
\begin{eqnarray}
\frac{d}{dt}(\tilde{E}[\hat{w}]^2)+\frac{c_{1}}{1+R^2}|\xi|^2|\hat{w}_{1}|^2\leq0, \label{R-E123}
\end{eqnarray}
where $\tilde{E}[\hat{w}]:=\{\phi^2E[\hat{w}]\}^{1/2}$ and $\tilde{E}[\hat{w}]\approx|\hat{w}_{1}|$.

Integrating (\ref{R-E123}) over $\mathbb{R}^{n}_{\xi}$, and using Plancherel's theorem gives
\begin{eqnarray}
\frac{d}{dt}\tilde{\mathcal{E}}_{1}^2+c_{3}\|\nabla w_{1}\|^2_{L^2}\leq 0, \label{R-E124}
\end{eqnarray}
where $$\mathcal{\tilde{E}}_{1}:=\Big(\int_{R^{n}_{\xi}}\tilde{E}[\hat{w}]^2d\xi\Big)^{1/2}\approx\|w_{1}\|_{L^2}$$
and the constant $c_{3}>0$ depends on $R$.
\\

\noindent\textit{\underline{Step 3.}}
Applying the operator $\dot{\Delta}_{q}(q\in \mathbb{Z})$ to (\ref{R-E111}) and performing the inter product with
$\dot{\Delta}_{q}w$, we can infer that
\begin{eqnarray}
\|\dot{\Delta}_{q}w\|_{L^2}+\|(I-\mathcal{P})\dot{\Delta}_{q}w\|_{L^2_{t}(L^2)}\leq\|\dot{\Delta}_{q}w\|_{L^2}, \label{R-E125}
\end{eqnarray}
which implies that
\begin{eqnarray}
\|w\|_{\dot{B}^{-s}_{2,\infty}}\leq\|w_{0}\|_{\dot{B}^{-s}_{2,\infty}}.  \label{R-E126}
\end{eqnarray}

\noindent\textit{\underline{Step 4.}}
Noticing that Lemma \ref{lem8.2}, we have (taking $k=0$ and $\varrho=s$)
\begin{eqnarray}
\|f\|_{L^2} \lesssim\|\nabla f\|^{\theta}_{L^2}\|f\|^{1-\theta}_{\dot{B}^{-s}_{2,\infty}}\ \ \  \Big(\theta=\frac{s}{1+s}\Big). \label{R-E127}
\end{eqnarray}
Applying (\ref{R-E127}) to the low-frequency part $w_{1}$, furthermore, we obtain the differential equality from (\ref{R-E124}):
\begin{eqnarray}
\frac{d}{dt}\tilde{\mathcal{E}}_{1}^2+C\|w_{0}\|_{\dot{B}^{-s}_{2,\infty}}^{-2/s}\|w_{1}\|_{L^2}^{2(1+1/s)}\leq0,   \label{R-E128}
\end{eqnarray} where we used the simple fact $\|w_{10}\|_{\dot{B}^{-s}_{2,\infty}}\leq\|w_{0}\|_{\dot{B}^{-s}_{2,\infty}}.$

Solve the differential inequality (\ref{R-E128}) to get
\begin{eqnarray}
\|w_{1}\|_{L^2}\lesssim\|w_{0}\|_{\dot{B}^{-s}_{2,\infty}}(1+t)^{-s/2}. \label{R-E129}
\end{eqnarray}
Finally, together with (\ref{R-E122}) and (\ref{R-E129}), we arrive at the decay estimate (\ref{R-E112}). Furthermore,
(\ref{R-E113}) is followed from (\ref{R-E112}) and the embedding  $L^p(\mathbb{R}^{n})\hookrightarrow\dot{B}^{-s}_{2,\infty}(\mathbb{R}^{n})(s=n(1/p-1/2))$.

Hence, the proof of Proposition \ref{prop7.1} is completed.
\end{proof}

\subsection{The hyperbolic-parabolic composite system}

Consider the following Cauchy problem for hyperbolic-parabolic composite system
\begin{equation}
\left\{
\begin{array}{l}
A^{0}U_{t} + \sum_{j=1}^{n}A^{j}U_{x_{j}}= \sum_{j,k}^{n}B^{j,k}U_{x_{j}x_{k}} , \\
U(0,x)=U_{0}.
\end{array} \right.\label{R-E130}
\end{equation}
Here $U:=(u,v)^{\top}$ where $u(t,x)$ and $v(t,x)$ are vectors with $N_{1}$ and $N_{2}$ components, respectively.
Set $N=N_{1}+N_{2}$. $A^{j}(j=0,\cdot\cdot\cdot,n), B^{j,k}$ are real constant matrices of order $N$.
where $\tilde{B}^{j,k}$ is the real constant matrix of order $N_{2}$.

We assume that the equation of (\ref{R-E130}) is ``symmetric hyperbolic-parabolic" in the sense that
\begin{itemize}
\item[(A3)] Matrices $A^{j}(j=0,\cdot\cdot\cdot,n)$ and $B^{j,k}$ are real symmetric and, in addition, $A^{0}$ is positive
definite. $B(\omega)\equiv\sum_{jk}B^{j,k}\omega_{j}\omega_{k}$ is nonnegative definite for any $\omega\in \mathbb{S}^{n-1}$ and and its null space coincides with $\mathcal{M}$.
\end{itemize}
Taking the
Fourier transform on (\ref{R-E130}) with respect to $x\in
\mathbb{R}^{n}$, we obtain
\begin{equation}
A^{0}\hat{U}_{t}+i|\xi|A(\omega)\hat{U}+|\xi|^2B(\omega)\hat{U}=0,
\label{R-E1000}
\end{equation}
where $A(\omega):=\sum_{j=1}^{n}A^{j}\omega_{j}$.
 Let $\lambda=\lambda(i\xi)$ be the eigenvalues of
(\ref{R-E1000}), which solves the characteristic equation
$$\mathrm{det}(\lambda A^{0}+i|\xi|A(\omega)+|\xi|^2B(\omega))=0.$$

Similar to Definition \ref{defn2.4}, we assume the following [SK] condition for (\ref{R-E130}).
\begin{itemize}
\item[(A4)] Let $\phi\in
\mathbb{R}^{N}$ satisfies $\phi\in\mathcal{M}$ (i.e., $B(\omega)\phi=0$) and
$\lambda A^{0}+A(\omega)\phi=0$ for some
$(\lambda,\omega)\in \mathbb{R}\times\mathbb{S}^{n-1}$, then
$\phi=0$.
\end{itemize}

Furthermore, it was shown by \cite{SK} that the [SK] condition
have the following equivalent characterization.

\begin{thm}\label{thm7.2} The following statements are
equivalent to each other.
\begin{itemize}
\item [$(\bullet)$] The system (\ref{R-E130}) satisfies the [SK] stability condition in (A4);
\item [$(\bullet)$] $\mathrm{Re}\lambda(i\xi)<0$ for $\xi\neq0$;
\item [$(\bullet)$] There is a constant $c>0$ such that $\mathrm{Re}\lambda(i\xi)\leq-c|\xi|^2/(1+|\xi|^2)$ for $\xi\in\mathbb{R}^{n}$;
\item [$(\bullet)$] There is an $N\times N$ matrix $K(\omega)$ depending smooth on $\omega\in\mathbb{S}^{n-1}$ satisfying the
properties:
\begin{itemize}
\item [(i)] $K(-\omega)=-K(\omega)$ for
$\omega\in\mathbb{S}^{n-1}$; \item [(ii)] $K(\omega)A^{0}$
is
skew-symmetric for $\omega\in\mathbb{S}^{n-1}$;
\item [(iii)]$[K(\omega)A(\omega)]'+B(\omega)$ is positive definite for
$\omega\in\mathbb{S}^{n-1}$.
\end{itemize}
\end{itemize}
\end{thm}

Having the above preparation,  we also obtain the decay estimates for (\ref{R-E130}).
\begin{prop}\label{prop7.2}
Let the assumptions (A3)-(A4) hold. Suppose $U_{0}\in L^2(\mathbb{R}^{n})\\\cap \dot{B}^{-s}_{2,\infty}(\mathbb{R}^{n})$ for $s>0$, then the solution of (\ref{R-E130}) has the decay rate
\begin{equation}
\|U\|_{L^2(\mathbb{R}^{n})}\lesssim (1+t)^{-s/2}. \label{R-E131}
\end{equation}
In particular, suppose $U_{0}\in L^2(\mathbb{R}^{n})\cap L^p(\mathbb{R}^{n})(1\leq p<2$), one further has
\begin{equation}
\|w\|_{L^2(\mathbb{R}^{n})}\lesssim (1+t)^{-\frac{n}{2}(\frac{1}{p}-\frac{1}{2})}. \label{R-E132}
\end{equation}
\end{prop}

\begin{proof}
The proof is just similar to that of Proposition \ref{prop7.1}. In order to explain how the strong dissipative term
plays a key role, we give the outline of the proof.

Firstly, there exists a constant $c_{0}>0$ such that
\begin{equation}
\frac{1}{2}\frac{d}{dt}(A^{0}\hat{U},\hat{U})+c_{0}|\xi|^2|\hat{v}|^2\leq 0, \label{R-E133}
\end{equation}
since $A^{0}, A(\omega)$ are real symmetric and $\tilde{B}^{j,k}$ is real symmetric and positive definite.
Then it follows from the [SK] condition that there exist two constants $c_{1}\geq0$ and $c_{2}>0$ such that

\begin{eqnarray}
\frac{d}{dt}E[\hat{U}]+c_{1}|\xi|^2|\hat{v}|^2+\frac{c_{2}|\xi|^2}{1+|\xi|^2}|\hat{U}|^2\leq0\label{R-E134}
\end{eqnarray}
with
\begin{eqnarray*}
E[\hat{U}]=\frac{1}{2}(A^{0}\hat{U},\hat{U})+\frac{\kappa}{2}\mathrm{Im}\Big(\frac{|\xi|}{1+|\xi|^2}K(\omega)A^{0}\hat{U},\hat{U}\Big),
\end{eqnarray*}
where $\kappa>0$ is chosen so small that $E[\hat{U}]\approx\hat{U}^2$.

Just doing the same high-frequency and low-frequency decomposition as (\ref{R-E121})-(\ref{R-E124}), we arrive at
\begin{eqnarray}
\|U_{2}\|_{L^2}\leq Ce^{-c_{3}t}\|U_{0}\|_{L^2},\label{R-E135}
\end{eqnarray}
and
 \begin{eqnarray}
\frac{d}{dt}\mathcal{\tilde{E}}_{2}^2+c_{4}\|\nabla U_{1}\|^2_{L^2}\leq 0, \label{R-E136}
\end{eqnarray}
where $U_{1}:=\mathcal{F}^{-1}[\phi\hat{U}], U_{2}:=\mathcal{F}^{-1}[\varphi\hat{U}]$ and $U=U_{1}+U_{2}$, additionally, $$\mathcal{\tilde{E}}_{2}:=\Big(\int_{R^{n}_{\xi}}\phi^2E[\hat{U}]^2d\xi\Big)^{1/2}\approx \|U_{1}\|_{L^2}$$
and the constants $c_{3},c_{4}>0$ both depend on $R$.



Finally, with the aid of Lemma \ref{lem8.2}, we can obtain
\begin{eqnarray}
\frac{d}{dt}\mathcal{\tilde{E}}_{2}^2+C\|U_{0}\|_{\dot{B}^{-s}_{2,\infty}}^{-2/s}\|U_{1}\|^{2(1+1/s)}_{L^2}\leq0,\label{R-E139}
\end{eqnarray}
which leads to the desired decay estimates (\ref{R-E131}) and (\ref{R-E132}).\end{proof}


\begin{rem}\label{rem7.1}
The low-frequency and high-frequency decomposition enable us to capture the effective pointwise behavior of the pseudo-differential operator $(1-\Delta)^{-\frac{1}{2}}\nabla$, which is a crucial point in our analysis. It is the first time to obtain the decay estimates for the low-frequency part of solutions by using the interpolation inequality related to the Besov space $\dot{B}^{-s}_{2,\infty}$, which is
quite different in comparison with the classical manner as in \cite{UKS}. Furthermore,  the low-frequency and high-frequency techniques  inspire us to use the Littlewood-Paley decomposition and do more elaborate analysis, which provides the motivation on studying decay problems on the framework of spatially Besov spaces for general damped hyperbolic system or hyperbolic-parabolic composite system.
\end{rem}


\section{Appendix B (interpolation inequalities)}\setcounter{equation}{0}\label{sec:B}
For the convenience of reader, we present some interpolation inequalities, actually,
which parallel the work of \textsc{Sohinger} and \textsc{Strain} \cite{SS}. However, we make some simplicity for use,
since their inequalities are related to the mixed spaces containing the microscopic velocity.
\begin{lem}\label{lem8.1}
Suppose $k\geq0$ and $m,\varrho>0$. Then the following inequality holds
\begin{eqnarray}
\|f\|_{\dot{B}^{k}_{2,1}}\lesssim \|f\|^{\theta}_{\dot{B}^{k+m}_{2,\infty}}\|f\|^{1-\theta}_{\dot{B}^{-\varrho}_{2,\infty}} \ \  \mbox{with}\ \ \ \theta=\frac{\varrho+k}{\varrho+k+m}. \label{R-E145}
\end{eqnarray}
\end{lem}

As a matter of fact, the interpolation inequality (\ref{R-E145}) can hold true for the general case $1\leq p\leq\infty$ rather than $p=2$ only. In particular, we have
\begin{lem}\label{lem8.2}
Suppose $k\geq0$ and $m,\varrho>0$.  Then the following inequality holds
\begin{eqnarray}
\|\Lambda^{k}f\|_{L^2}\lesssim \|\Lambda^{k+m}f\|^{\theta}_{L^2}\|f\|^{1-\theta}_{\dot{B}^{-\varrho}_{2,\infty}} \ \  \mbox{with}\ \ \ \theta=\frac{\varrho+k}{\varrho+k+m}, \label{R-E146}
\end{eqnarray}
where (\ref{R-E146}) is also true for $\partial^{\alpha}$ with $|\alpha|=k(k$ nonnegative integer).
\end{lem}
The proof follows from the embedding
$\dot{B}^{k}_{2,1}\hookrightarrow\dot{B}^{k}_{2,2}\hookrightarrow\dot{B}^{k}_{2,\infty}$ and $\|f\|_{\dot{B}^{k}_{2,2}}\approx\|\Lambda^{k}f\|_{L^2}$ directly.

\begin{lem}\label{lem8.3}
Suppose that $m\neq\varrho$. Then the following inequality holds
\begin{eqnarray}
\|f\|_{\dot{B}^{k}_{p,1}}\lesssim \|f\|^{1-\theta}_{\dot{B}^{m}_{r,\infty}}\|f\|^{\theta}_{\dot{B}^{\varrho}_{r,\infty}}, \label{R-E147}
\end{eqnarray}
where $0<\theta<1$,  $1\leq r\leq p\leq\infty$ and $$k+n\Big(\frac{1}{r}-\frac{1}{p}\Big)=m(1-\theta)+\varrho\theta.$$
\end{lem}

The proof is just the same one as in \cite{SS} without the mixed Hilbert space related to the microscopic velocity.
\begin{lem}\label{lem8.4}
Suppose that $m\neq\varrho$. One has the interpolation inequality of Gagliardo-Nirenberg-Sobolev type
\begin{eqnarray}
\|\Lambda^{k}f\|_{L^{q}}\lesssim \|\Lambda^{m}f\|^{1-\theta}_{L^{r}}\|\Lambda^{\varrho}f\|^{\theta}_{L^{r}}, \label{R-E148}
\end{eqnarray}
where $0\leq\theta\leq1$,  $1\leq r\leq q\leq\infty$ and $$k+n\Big(\frac{1}{r}-\frac{1}{q}\Big)=m(1-\theta)+\varrho\theta.$$
\end{lem}
\begin{proof}
Indeed, it is the direct consequence of Lemma \ref{lem8.3}.
Due to the homogeneous decomposition $f=\sum_{q\in \mathbb{Z}}\dot{\Delta}_{q}f$, we have
\begin{eqnarray}
\|\Lambda^{k}f\|_{L^{q}}&=&\|\sum_{q\in \mathbb{Z}}\dot{\Delta}_{q}\Lambda^{k}f\|_{L^{q}}\leq\sum_{q\in \mathbb{Z}}\|\dot{\Delta}_{q}\Lambda^{k}f\|_{L^{q}}\nonumber\\ &\approx& \sum_{q\in \mathbb{Z}} 2^{qk}\|\dot{\Delta}_{q}f\|_{L^{q}}=\|f\|_{\dot{B}^{k}_{q,1}},\label{R-E149}
\end{eqnarray}
where we have used the Minkowski's inequality and Bernstein's inequality (see Lemma \ref{lem3.1}).

In the other hand, note that the embedding $L^{r}\hookrightarrow\dot{B}^{0}_{r,\infty}(1\leq r\leq\infty)$ and
the Bernstein's inequality, we have
\begin{eqnarray}
\|f\|_{\dot{B}^{m}_{r,\infty}}\approx\|\Lambda^{m}f\|_{\dot{B}^{0}_{r,\infty}}\lesssim \|\Lambda^{m}f\|_{_{L^r}}. \label{R-E150}
\end{eqnarray}
Hence, (\ref{R-E148}) is followed from Lemma \ref{lem8.3} immediately.
\end{proof}

\begin{lem}\label{lem8.5}
Suppose that $\varrho>0$ and $1\leq p<2$. One has
\begin{eqnarray}
\|f\|_{\dot{B}^{-\varrho}_{r,\infty}}\lesssim \|f\|_{L^{p}} \label{R-E151}
\end{eqnarray}
with $1/p-1/r=\varrho/n$. In particular, this holds with $\varrho=n/2, r=2$ and $p=1$.
\end{lem}
\begin{proof}
It follows from Bernstein's inequality and Young's inequality that
\begin{eqnarray}
2^{q\frac{n}{r}}\|\dot{\Delta}_{q}f\|_{L^r}\lesssim 2^{q\frac{n}{p}}\|\dot{\Delta}_{q}f\|_{L^p}\lesssim 2^{q\frac{n}{p}}\|f\|_{L^p} \label{R-E152}
\end{eqnarray}
for all $q\in \mathbb{Z}$, which implies (\ref{R-E151}) directly.
\end{proof}

\section*{Acknowledgments}
The authors are very grateful to Prof. R.-J Duan for his
warm communication. The first author (J. Xu) is partially supported by the Program for New Century Excellent
Talents in University (NCET-13-0857), Special Foundation of China Postdoctoral
Science Foundation (2012T50466) and the NUAA Fundamental
Research Funds (NS2013076). He would like to thank Prof. Kawashima
for giving him an opportunity to work at Kyushu University.  The work is also partially supported by
Grant-in-Aid for Scientific Researches (S) 25220702 and (A) 22244009.

\end{document}